\RequirePackage{fix-cm}
\documentclass[smallextended]{svjour3}       
\smartqed  
\usepackage{graphicx}
\usepackage{amsmath}
\usepackage{amssymb}
\usepackage[shortlabels]{enumitem}
\usepackage{url}

\newcommand{\bbd}{{\ooalign{$d$\cr $\mkern6.8mul$}}}
\newcommand{\dblarrow}{\mathbin{\ooalign{$\scriptstyle\rightarrow$\cr\raise.75ex\hbox{$\scriptstyle\rightarrow$}}}}
\DeclareMathOperator*{\aslim}{as-lim}
\DeclareMathOperator*{\asliminf}{as-lim\,inf}
\DeclareMathOperator*{\aslimsup}{as-lim\,sup}

\DeclareMathOperator*{\earg}{\epsilon-arg}
\DeclareMathOperator*{\zarg}{-arg}
\DeclareMathOperator*{\elim}{e-lim}

 \usepackage{mathptmx}      
\DeclareMathAlphabet{\mathcal}{OMS}{cmsy}{m}{n}
\journalname{Mathematical Programming}
\begin{document}

\title{Statistics with Set-Valued Functions\thanks{This work was supported in part by NSF Award CMMI-1450963.}
}
\subtitle{Applications to Inverse Approximate Optimization}


\author{Anil Aswani
}


\institute{A. Aswani \at
              Industrial Engineering and Operation Research, University of California, Berkeley, CA, USA\\
              \email{aaswani@berkeley.edu}           
}


\maketitle

\begin{abstract}
Much of statistics relies upon four key elements: a law of large numbers, a calculus to operationalize stochastic convergence, a central limit theorem, and a framework for constructing local approximations.  These elements are well-understood for objects in a vector space (e.g., points or functions); however, much statistical theory does not directly translate to sets because they do not form a vector space.  Building on probability theory for random sets, this paper uses variational analysis to develop operational tools for statistics with set-valued functions.  These tools are first applied to nonparametric estimation (kernel regression of set-valued functions).  The second application is to the problem of inverse approximate optimization, in which approximate solutions (corrupted by noise) to an optimization problem are observed and then used to estimate the amount of suboptimality of the solutions and the parameters of the optimization problem that generated the solutions.  We show that previous approaches to this problem are statistically inconsistent when the data is corrupted by noise, whereas our approach is consistent under mild conditions.

\end{abstract}

\section{Introduction}

While statistical theory is well-developed for problems concerning (single-valued) functions \cite{bickel2006,van2000}, there has been less work on statistics with sets or set-valued functions.  Most attention in statistics on sets has been focused on the problem of estimating a single set under different measurement models \cite{devroye1980,geffroy1964,guntuboyina2012,korostelev1995,patschkowski2016,renyi1963,scholkopf2001}.  The problem of estimating set-valued functions is less well studied, though it has potential applications in varied domains including healthcare, robotics, and energy.  For instance, we study in this paper the problem of inverse approximate optimization, where approximate solutions (corrupted by noise) to a parametric optimization problem are observed and then used to estimate the amount of suboptimality of the solutions and the parameters that generated the solutions.  Inverse approximate optimization can be used to construct predictive models of human behavior and decision-making, where the explicit model is that an individual makes decisions by approximately solving an optimization problem.  Statistical estimation in this context could be used to quantify the tradeoffs made by a particular individual between competing objectives, as well as quantify the predictability of the decision-making process.  This particular problem of inverse approximate optimization is related to the broader topic of statistics with set-valued functions because the solution mapping of an (even strictly convex) optimization problem becomes a set when suboptimality of solutions is allowed.  Thus a framework for statistics with set-valued functions is needed to study such problems.

A substantial impediment to studying such estimation problems is the lack of statistical tools for random sets and set-valued functions, and two technical issues prevent the use of existing tools.  The first is that most statistical theory assumes objects belong to a vector space, which is the case for points and functions.  But sets do not form a vector space, and so existing statistical theory cannot be used.  This is a fundamental difficulty, and even the usual notion of expectation does not apply to sets \cite{molchanov2006}.  The second is that most statistical theory has been developed by using metrics and distance functions to derive results. But analyzing sets using distances is difficult, and most analysis tools and results for sets do not use this approach \cite{berge1963,rockafellar2009}.

Arguably the most natural approach to statistics with random sets is to define a family of sets parametrized by a random vector, and then perform standard statistical analysis with respect to this parametrization.  However it is not clear without further analysis whether stochastic convergence of the estimated parameters implies stochastic convergence of the corresponding set estimates.  We study this question in a more general framework and give a counterexample to demonstrate how parameter convergence does not always imply set convergence.  Moreover, the parametrization approach does not lead to a useful definition for the expectation of random sets \cite{molchanov2006}; the reason is that the expectation of the parameters does not characterize the expectation of the set in a way in that ensures the law of large numbers holds.

One goal of this paper is to establish tools for statistics with set-valued functions, and this requires understanding four main ingredients: a law of large numbers, a calculus to operationalize stochastic convergence, a central limit theorem, and tools for constructing local approximations.  Probability theory for random sets \cite{molchanov2006} provides an expectation for random sets \cite{aumann1965,kudo1953}, a law of large numbers \cite{artstein1975}, and a central limit theorem \cite{weil1982}.  Here we use variational analysis \cite{rockafellar2009} to advocate a notion of local approximation for set-valued functions, and to develop results that allow us to interpret stochastic convergence and expectations of random sets as operators.

The paper begins by describing our notation and providing some useful definitions related to set-valued functions.  We focus in this paper on almost sure (a.s.) convergence because the corresponding definitions and approach most clearly demonstrate the tight link between variational analysis and statistics.  Defining set convergence in probability requires metrization, which partially obscures the relationship to variational analysis.  We also focus on Lipschitz continuity for set-valued functions because we advocate using this concept as a notion of local approximation for set-valued functions.  The utility of this approach is displayed later in the paper when we use Lipschitz continuity as a replacement for differentiability when proving a Delta method-like result and proving statistical consistency of a kernel regression estimator.

The next section shows how to interpret stochastic convergence and expectation of random sets as operators.  We study the limit of sequences of sets under different set operations, after proving a set-based generalization of the continuous mapping theorem \cite{bickel2006} from statistics.  Then we study the expectation of random sets under various set operations.  Standard proofs about the properties of the expectation of random variables do not extend because the expectation of a random set cannot be computed by integration.  This means properties like distribution of expectation under independence of the product of a random matrix with a random set or Jensen's inequality have not been previously established, and we prove such results.  We conclude by reviewing a law of large numbers and a central limit theorem for random sets.

Another goal of this paper is to study two problems of estimating set-valued functions, and through the process of analyzing these problems we demonstrate the utility of our tools for statistics with set-valued functions.  The first problem we study is estimating a set-valued function using noisy measurements of the set.  We propose a kernel regression estimator that can be interpreted as a generalization of methods for functions \cite{aswani2011,aswani2013,bickel1982,noda1976,wand1994}.  The key step in proving statistical consistency is using Lipschitz continuity of the set-valued function to construct local approximations.  We show that statistical consistency follows by combining our results on stochastic convergence with convergence bounds on (vector-valued) random variables. 

The second problem we study is inverse approximate optimization, where noisy measurements of approximate solutions to an optimization problem are used to estimate the suboptimality of the solutions and the parameters of the optimization problem.  In contrast, past work on inverse optimization assumes no noise \cite{ahuja2001,chan2014,esfahani2015} or exact solutions \cite{aswani2015,bertsimas2015,keshavarz2011}.  We develop a method for inverse approximate optimization and prove its statistical consistency using stochastic epi-convergence \cite{aswani2011,dupacova1988,geyer1994,knight2000,lachout2005,salinetti1986}.  Combining with our results on stochastic convergence and results on the continuity of solutions to optimization problems \cite{rockafellar2009,royset2016b} shows our method consistently estimates the (set-valued) approximate solution mapping that generates the data.

We conclude by examining extensions of the problem of inverse approximate optimization, as well as discussing related open questions about statistics with set-valued functions.  In particular, we describe how some extensions lead to formulations of optimization problems with structures (e.g., objective functions that are integrals whose domain of integration depends on the decision variable) that have not been well-studied from the perspective of numerical optimization.  Performing statistics with sets and set-valued functions also leads to questions about the design of numerical representations of sets.  We argue that further study of statistics with set-valued functions will require developing new numerical methods and optimization theory.

\section{Preliminaries}

This section presents the notation used in this paper, as well as several useful concepts from variational analysis.  Most of the variational analysis definitions are from \cite{rockafellar2009}. The definition of set-valued set functions is from \cite{matheron1975}, and we use the definitions of the Minkowski set operations from \cite{schneider1993}.  We abbreviate \emph{almost surely} using a.s.

\subsection{Notation}

Let $\mathcal{F}(E)$ be the space of closed subsets of $E$, and let $\mathcal{K}(E)$ be the space of compact subsets of $E$.  We will focus on cases where $E$ is a Euclidean space, and so will use the notation $\mathcal{F},\mathcal{K}$ to refer to the corresponding spaces.  Clearly $\mathcal{F}\supset\mathcal{K}$ by definition.

Suppose $C,D$ are sets and $\Psi$ is a matrix or scalar. We use the set notation: $C\cup D$ is the union of $C,D$; $C\cap D$ is the intersection of $C,D$; $C \subseteq D$ denotes that $C$ is a subset of $D$; $C \supseteq D$ denotes that $C$ is a superset of $D$; $\mathrm{cl}(C)$ is the closure of $C$; $\mathrm{co}(C)$ is the convex hull of $C$; $C^{\ \mathsf{c}}$ is the complement of $C$; $\partial C$ is the boundary of $C$; $C\oplus D = \{c + d : c\in C, d\in D\}$ is the Minkowski sum of $C,D$; $C\ominus D = \{x : x\oplus D\subseteq C\}$ is the Minkowski difference of $C,D$; $\Psi\cdot C = \{\Psi\cdot c : c\in C\}$; and $\Psi^{-1}C = \{\Psi^{-1}\cdot c : c\in C\}$.


\subsection{Limit Definitions and Set-Valued Mappings}

The outer limit of the sequence of sets $C_n$ is defined as
\begin{equation}
\textstyle\limsup_n C_n = \{x : \exists n_k \text{ s.t. } x_{n_k} \rightarrow x \text{ with } x_{n_k}\in C_{n_k}\},
\end{equation}
and the inner limit of the sequence of sets $C_n$ is defined as
\begin{equation}
\textstyle\liminf_n C_n = \{x : \exists x_n \rightarrow x \text{ with } x_n\in C_n\}.
\end{equation}
The outer limit consists of all the cluster points of $C_n$, whereas the inner limit consists of all limit points of $C_n$.  The limit of the sequence of sets $C_n$ exists if the outer and inner limits are equal, and we define that $\textstyle\lim_n C_n := \limsup_n C_n = \liminf_n C_n$.

Let $\overline{\mathbb{R}} = [-\infty,\infty]$ denote the extended real line.  A sequence of extended-real-valued functions $f_n : X\rightarrow\overline{\mathbb{R}}$ is said to epi-converge to $f$ if at each $x\in X$ we have
\begin{equation}
\begin{aligned}
\begin{cases}
\lim\inf_n f_n(x_n)\geq  f(x) & \text{for every sequence } x_n\rightarrow x\\
\lim\sup_n f_n(x_n)\leq f(x) &\text{for some sequence }x_n\rightarrow x
\end{cases}
\end{aligned}
\end{equation}
The notion of epi-convergence is so-named because it is equivalent to set convergence of the epigraphs of $f_n$, meaning that epi-convergence is equivalent to the condition $\lim_n \{(x,\alpha)\in X\times\mathbb{R} : f_n(x) \leq \alpha\} = \{(x,\alpha)\in X\times\mathbb{R} : f(x) \leq \alpha\}$.

A set-valued set function $G : V \Rightarrow U$ assigns to each set $S \subseteq V$ a set $G(S) \subseteq U$.  The outer limit of $G$ at the set $\overline{S}\in V$ is defined as
\begin{equation}
\limsup_{S\rightarrow\overline{S}} G(S) = \{u : \exists S_n \rightarrow \overline{S} \text{ s.t. } u_n \rightarrow u \text{ with } S_n\subseteq V, u_n\in G(S_n)\},
\end{equation}
and the inner limit of $G$ at the set $\overline{S}\subseteq V$ is defined as
\begin{equation}
\liminf_{S\rightarrow\overline{S}} G(S) = \{u : \forall S_n \rightarrow \overline{S},\ \exists u_n \rightarrow u \text{ with } S_n\subseteq V, u_n\in G(S_n)\}.
\end{equation}
The intuition is similar to the notions for sequences of sets.  The set-valued set function $G$ is outer semicontinuous (osc) at $\overline{S}$ if $\limsup_{S\rightarrow\overline{S}} G(S) \subseteq G(\overline{S})$, and $G$ is inner semicontinuous (isc) at $\overline{S}$ if $\liminf_{S\rightarrow\overline{S}} G(S) \supseteq G(\overline{S})$.  The set-valued set function $G$ is continuous at $\overline{S}$ when it is both osc and isc, that is when $\lim_{S\rightarrow\overline{S}}G(S) = G(\overline{S})$.

Variational analysis typically uses set-valued functions, rather than set-valued set functions. A set-valued function $F : X \dblarrow U$ assigns to each point $x \in X$ a set $F(x) \subseteq U$.  Outer limits, inner limits, outer semicontinuity, inner semicontinuity, and continuity are defined as above but with points replacing sets in the domain.  Moreover, a set-valued function applied pointwise to sets is an osc, isc, continuous set-valued set function whenever the set-valued function is osc, isc, continuous, respectively.

\subsection{Probability Definitions and Stochastic Convergence}

Let $(\Omega, \mathfrak{F}, \mathbb{P})$ be a complete probability space, where $\Omega$ is the sample space, $\mathfrak{F}$ is the set of events, and $\mathbb{P}$ is the probability measure.  A map $S : \Omega\rightarrow\mathcal{F}$ is a random set if $\{\omega : S(\omega) \in\mathcal{X}\}\in\mathfrak{F}$ for each $\mathcal{X}$ in the Borel $\sigma$-algebra on $\mathcal{F}$ \cite{molchanov2006}.  Like the usual convention for random variables, we notationally drop the argument for a random set.  

When discussing samples for estimation, we use the convention that capital letters denote random variables, and lowercase letters denote measured data. Also, we use the notation $U(a,b)$ to specify a uniform distribution with support $[a,b]$.

We next define almost sure stochastic convergence of random sets.  The notation $\aslimsup_n C_n \subseteq C$ denotes $\mathbb{P}(\limsup_n C_n\subseteq C) = 1$, the notation $\asliminf_n C_n \supseteq C$ denotes $\mathbb{P}(\liminf_n C_n\supseteq C) = 1$, and the notation $\aslim_n C_n = C$ denotes $\mathbb{P}(\lim_n C_n = C) = 1$.  Note $\aslimsup_n C_n \subseteq C$ and $\asliminf_n C_n \supseteq C$ if and only if $\aslim_n C_n = C$, since a countable intersection of almost sure events occurs almost surely.

\subsection{Distances and Lipschitz Continuity}

\label{sec:dlc}

Let $d(x,C) = \inf_{y\in C}\|x-y\|$ and $d^2(x,C) = \inf_{y\in C}\|x-y\|^2$ be the distance and squared distance, respectively, from a point $x$ to set $C$.  The support function of $C$ is $h(x,C) = \sup_{y\in C} x^\mathsf{T} y$.  We also define the indicator function $\delta(x,C)$ to equal $0$ when $x\in C$ and $+\infty$ when $x\notin C$.  The (integrated) set distance between $C$ and $D$ is defined as $\bbd(C,D) = \int_0^\infty \bbd_r(C,D)e^{-r}dr$, where the pseudo-distance between sets $C$ and $D$ is given by $\bbd_r(C,D) = \max_{\|x\|\leq r}\big|d(x,C) - d(x,D)\big|$. Note $\bbd_r(\{x\},C)\ne d(x,C)$ for all $r$.  The integrated set distance $\bbd$ is a metric that characterizes the convergence defined earlier for sets in $\mathcal{F}$, and the Pompeiu-Hausdorff distance $\bbd_\infty$ is a metric that characterizes the convergence defined earlier for sets in $\mathcal{K}$.  Since these metrics are complex, the sequence characterization of convergence is arguably more natural for sets.

One exception to this statement is in defining Lipschitz continuity for set-valued functions.  A set-valued function $F : X \dblarrow U$ is Lipschitz continuous on $X$ with constant $\kappa\in\mathbb{R}_+$ if it is nonempty, closed-valued and such that
\begin{equation}
F(x') \subseteq F(x) + \kappa\|x'-x\|\mathbb{B}\  \text{ for } x,x'\in X,
\end{equation}
where $\mathbb{B} = \{u : \|u\|\leq 1\}$ is the unit ball.  A set-valued set function $G : V \Rightarrow U$ is Lipschitz continuous on $V$ with constant $\kappa\in\mathbb{R}_+$ if it is nonempty, closed-valued and 
\begin{equation}
G(S') \subseteq G(S) + \kappa\bbd_\infty(S',S)\mathbb{B}\  \text{ for } S,S'\subseteq V \text{ with } S,S' \in\mathcal{K}.
\end{equation}
We will make use of Lipschitz continuity as a zeroth-order local approximation.

\section{Mathematical Tools for Statistics with Set-Valued Functions}

This section develops mathematical tools that allow us to interpret stochastic convergence and the expectation of random sets as operators.  We prove results on the limit of sequences of sets under different set operations, define an expectation for random sets, and then derive results about the behavior of this expectation under different set operations.  We conclude this section by briefly summarizing a law of large numbers and a central limit theorem for random sets.

\subsection{Stochastic Limit Theorems}


Our reason for considering set-valued set functions is this allows us to more precisely generalize the classical \emph{continuous mapping theorem} of statistics \cite{bickel2006} to mappings applied to sequences of sets.  Because semicontinuity is an important aspect of set convergence, a generalization that considers semicontinuity leads to a richer set of results than simply considering continuity.

\begin{theorem}[Semicontinuous Mapping Theorem]
\label{thm:smt}
Let $G$ be a set-valued set function, and suppose $\aslim_n C_n = C$.  There are three cases:
\begin{enumerate}[$\mathrm{(\mathtt{\alph*})}$, leftmargin=2.5em]
\item If $G$ is osc at $C$, then $\aslimsup_n G(C_n) \subseteq G(C)$.
\item If $G$ is isc at $C$, then $\asliminf_n G(C_n) \supseteq G(C)$.
\item If $G$ is continuous at $C$, then $\aslim_n G(C_n) = G(C)$.
\end{enumerate}
\end{theorem}

\begin{proof}
The definition of osc (isc) means $\lim_n C_n = C$ implies $\limsup_n G(C_n) \subseteq G(C)$ ($\liminf_n G(C_n) \subseteq G(C)$).  This means $\mathbb{P}(\limsup_n G(C_n) \subseteq G(C)) \geq \mathbb{P}(\lim_n C_n = C) = 1$ ($\mathbb{P}(\liminf_n G(C_n) \supseteq G(C)) \geq \mathbb{P}(\lim_n C_n = C) = 1$), which shows the first two cases.  The third case follows from the first two cases by recalling that continuity at $C$ is equivalent to being both osc and isc at $C$.\qed
\end{proof}

\begin{remark}
One consequence is that the set-valued function $S(\theta)$ parametrized by $\theta$ has the behavior that $\aslim_n\theta_n=\theta_0$ implies $\aslim_nS(\theta_n)= S(\theta_0)$ only when the set is continuous with respect to the parametrization.  For example, consider $S(\theta) = \{1\}$ if $\theta > 0$, $S(\theta) = \{-1\}$ if $\theta < 0$, and $S(\theta) = [-1,1]$ if $\theta = 0$.  If $\theta_n = 1/n$, then $S(\theta_n) \equiv \{1\}$ and so $\aslim_n S(\theta_n) = \{1\}$.  But $\aslim_n\theta_n = 0$ and $S(0) = [-1,1]$.
\end{remark}

As is customary in statistics, we immediately get some useful corollaries to our semicontinuous mapping theorem by applying the theorem to specific mappings.  Our first corollary applies the semicontinuous mapping theorem to set operations like unions and intersections of sets, the boundary of sets, the convex hull of sets, etc.

\begin{corollary}
Let $C_n,D_n\in\mathcal{F}$ be almost surely convergent sequences of sets (i.e., $\aslim_n C_n = C$ and $\aslim_n D_n = D$).  Then we have:
\begin{enumerate}[$\mathrm{(\mathtt{\alph*})}$, leftmargin=2.5em]
\item $\aslim_n (C_n\cup D_n) = C\cup D$
\item $\aslimsup_n (C_n\cap D_n) \subseteq C\cap D$
\item $\asliminf_n \mathrm{cl}(C_n^{\ \mathsf{c}})\supseteq\mathrm{cl}(C^\mathsf{c})$
\item $\asliminf_n\partial C_n \supseteq \partial C$
\item $\asliminf_n \mathrm{co}(C_n) \supseteq \mathrm{co}(C)$
\item $\aslim_n \mathrm{co}(C_n) = \mathrm{co}(C)$, when there is a deterministic $C_0\in\mathcal{K}$ so $C_n \subseteq C_0$ a.s.
\end{enumerate}
\end{corollary}

\begin{proof}
We interpret $\cup,\cap$ as set-valued set functions with a domain over the product space $\mathcal{F}\times\mathcal{F}$: The function $G_1(S,T) = S\cup T$ is continuous \cite{matheron1975}, and the function $G_2(S,T) = S\cap T$ is osc \cite{matheron1975}.  The set complement and boundary operators can be interpreted as set-valued set functions with domain $\mathcal{F}$: The function $G_3(S) = \mathrm{cl}(S^\mathsf{c})$ is isc \cite{matheron1975}, and the function $G_4(S) = \partial S$ is isc \cite{matheron1975}.  The convex hull operation can be cast as set-valued set functions: $G_5(S) = \mathrm{co}(S)$ is isc when the domain is $\mathcal{F}$, and $G_6(S) = \mathrm{co}(S)$ is continuous when the domain is $C_0$ \cite{matheron1975}. The results now follow from the corresponding parts of the semicontinuous mapping theorem. \qed
\end{proof}

\begin{remark}
Note the above result states that the stochastic limit of the convex hull operator is sensitive to the domain of the sequence of sets.
\end{remark}

We can also apply the semicontinuous mapping theorem to the Minkowski set operations. These results are useful for proving convergence of statistical estimators.

\begin{corollary}
 \label{cor:minklim}
Let $C_n,D_n\in\mathcal{F}$ be almost surely convergent sequences of sets (i.e., $\aslim_n C_n = C$ and $\aslim_n D_n = D$), and let $\Psi_n$ be an almost surely convergent (in the Frobenius norm) sequence of matrices or scalars (i.e., $\aslim_n \Psi_n = \Psi$).  If there exists a deterministic $D_0\in\mathcal{K}$ so $D_n \subseteq D_0$ a.s., then
\begin{enumerate}[$\mathrm{(\mathtt{\alph*})}$, leftmargin=2.5em]
\item $\aslim_n (C_n\oplus D_n) = C\oplus D$ \label{cor:minklim_osum}
\item $\aslimsup_n (C_n\ominus D_n) \subseteq C\ominus D$, when $D \neq \emptyset$
\item $\aslim_n \Psi_n\cdot D_n = \Psi\cdot D$  \label{cor:minklim_prod}
\item $\aslim_n \Psi^{-1}_nD_n^{\vphantom{-1}} = \Psi^{-1}_{\vphantom{n}}D$, when $\Psi$ is invertible \label{cor:minklim_inv}
\end{enumerate}
\end{corollary}

\begin{proof}
We interpret $\oplus,\ominus$ as set-valued set functions with a domain over the product space $\mathcal{F}\times\mathcal{K}$: The function $G_1(S,T) = S\oplus T$ is continuous \cite{matheron1975}, and the function $G_2(S,T) = S\ominus T$ is osc if $T\neq\emptyset$ \cite{matheron1975}.  So the first two results follow from Theorem \ref{thm:smt}.  The multiplication operation can be interpreted as a set-valued set function $G_3(S,T) = T\cdot S$ with domain over the product space $\mathbb{M}\times D_0$, where $\mathbb{M}$ is the space of matrices of appropriate dimension or the space of scalars.  We show it is continuous.  Suppose $G_3$ is not osc at $\overline{S}\times\overline{T}$; then there exist $T_n\rightarrow\overline{T}$, $S_n\rightarrow\overline{S}$, and $u_n\rightarrow\overline{u}$ with $T_n\in\mathbb{M}$, $S_n\in \mathcal{F}$, $u_n\in T_n\cdot S_n$, and $\overline{u}\notin \overline{T}\cdot\overline{S}$.  But by the definition of matrix-set (or scalar-set) multiplication there exists $v_n \in S_n$ with $u_n = T_n\cdot v_n$, and by the boundedness by assumption of $D_0$ there exist $n_k$ and $\overline{v}$ such that $v_{n_k} \rightarrow\overline{v}$ with $\overline{v}\in\overline{S}$, which is a contradiction since matrix-vector (or scalar-vector) multiplication is osc.  Thus $G_3$ is osc.  Next, we show $G_3$ is isc at $\overline{T}\cdot\overline{S}$: Consider any $\overline{x}\in\overline{S}$ and $u = \overline{T}\cdot\overline{x}$, and let $T_n,S_n$ be any sequences satisfying $T_n\rightarrow\overline{T}$ and $S_n\rightarrow\overline{S}$.  By the inner limit definition there exists $x_n\rightarrow\overline{x}$ with $x_n\in S_n$, and so $T_n\cdot x_n\rightarrow\overline{T}\cdot\overline{x}$ with $T_n\cdot x_n\in T_n\cdot S_n$.  So $G_3$ satisfies the definition of being isc at $\overline{T}\cdot\overline{S}$, and is continuous since it is also osc.  The third result follows from Theorem \ref{thm:smt}.  The fourth result is proved by noting Theorem \ref{thm:smt} implies $\aslim_n \Psi_n^{-1} = \Psi_{\vphantom{n}}^{-1}$ since the matrix inverse operation is continuous except at points of singularity, and so $\aslim_n \Psi^{-1}_nC_n^{\vphantom{-1}} = \Psi^{-1}_{\vphantom{n}}C$ by the third result. \qed
\end{proof}

Our final results on stochastic limits are not based on the semicontinuous mapping theorem, but are nevertheless useful for writing stochastic convergence proofs.

\begin{lemma}[Sandwich Lemma]
Let $L_n\in\mathcal{F}$ and $U_n\in\mathcal{F}$ be almost surely convergent sequences of sets (i.e., $\aslim_n L_n = L$ and $\aslim_n U_n = U$), and let $C_n\in\mathcal{F}$ be a sequence of sets.  Then we have
\begin{enumerate}[$\mathrm{(\mathtt{\alph*})}$, leftmargin=2.5em]
\item $\aslimsup_n C_n\subseteq U$, when $C_n \subseteq U_n$ a.s.
\item $\asliminf_n C_n\supseteq L$, when $C_n \supseteq L_n$ a.s.
\item $\aslim_n C_n = L = U$, when $L_n\subseteq C_n \subseteq U_n$ a.s. and $L = U$
\end{enumerate}
\end{lemma}

\begin{proof}
For the first two results, note $\aslimsup_n C_n\subseteq \aslimsup_nU_n=\aslim_nU_n=U$ and $\asliminf_n C_n\supseteq \asliminf_nL_n=\aslim_nL_n=L$.  The third result follows from the first two results and the definition of limit.  \qed
\end{proof}

This sandwich lemma is valuable for statistical analysis, and we next present a convergence result that is helpful in proving statistical consistency.

\begin{corollary}
\label{cor:invoplus}
Let $C_n, D_n\in\mathcal{F}$ be sequences of sets, with $D_n \subseteq r_n\mathbb{B}$ for a sequence $r_n\in\mathbb{R}_+$.  If $\aslim_n r_n = 0$ and $\aslim_n C_n\oplus D_n$ exists, then $\aslim_n C_n = \aslim_n C_n\oplus D_n$.
\end{corollary}

\begin{proof}
Consider any $\overline{c}\in\aslimsup_n C_n$, and note that by the outer limit definition there exist $n_k$ and $c_{n_k}\in C_{n_k}$ such that $c_{n_k}\rightarrow \overline{c}$.  Thus $c_{n_k} + d_{n_k}\rightarrow \overline{c}$ for any $d_n\in D_n$ since by assumption $d_n \rightarrow 0$.  This means $\aslimsup_n C_n \subseteq \aslimsup_n C_n\oplus D_n = \aslim_n C_n\oplus D_n$, where the equality holds since $\aslim_n C_n\oplus D_n$ exists.  Next, choose any $\overline{u}\in\aslim_n C_n\oplus D_n$.  By the inner limit definition there exists $u_n\in C_n\oplus D_n$ such that $u_n\rightarrow \overline{u}$, and so by the Minkowski sum definition there exist $c_n\in C_n$ and $d_n\in D_n$ such that $u_n = c_n + d_n$ or equivalently that $c_n = u_n - d_n$.  Since by assumption $d_n \rightarrow 0$, this means $c_n \rightarrow \overline{u}$. Thus $\asliminf_n C_n\supseteq \aslim_n C_n\oplus D_n$. The result follows by noting $\asliminf_n C_n \subseteq \aslimsup_n C_n$ always holds, and combining with the above.\qed
\end{proof}


%

\subsection{Expectation}

Because sets do not form a vector space, defining expectations for random sets is not straightforward. In fact, a number of different definitions have been proposed \cite{molchanov2006} that capture different features that might be desired for an expectation operation.  One particularly useful definition is the \emph{selection expectation} \cite{kudo1953,aumann1965}.   This definition for the expectation of random sets is the most well studied because it leads to a corresponding law of large numbers and central limit theorem \cite{molchanov2006}.

For a random set $X$, a selection $\xi$ is a (single-valued) random vector that almost surely belongs to $X$.  We say the selection $\xi$ is integrable if $\mathbb{E}\|\xi\|_1$ is finite, where $\|\cdot\|_1$ is the usual $\ell_1$-norm.  The selection expectation of a random set $X$ is defined as
\begin{equation}
\mathbb{E}(X) = \mathrm{cl}\{\mathbb{E}\xi : \xi \in \mathcal{S}^1(X)\},
\end{equation}
where $\mathcal{S}^1(X)$ is the set of all integrable selections of $X$.  The random set $X$ is called integrable if $\mathcal{S}^1(X) \neq \emptyset$, and note this property implies $X$ is almost surely non-empty.

The selection expectation is difficult to use because it cannot be computed by taking an integral, as is the case for expectations for objects in a vector space.  But since we assume $E$ is Euclidean space, the definition of the selection expectation simplifies and has a sharp characterization \cite{molchanov2006}: If the probability space is nonatomic and $X$ is a bounded and closed integrable random set, then $\mathbb{E}(X) = \{\mathbb{E}\xi : \xi \in \mathcal{S}^1(X)\}$ is a compact set, $\mathbb{E}(X)$ is convex, $\mathbb{E}(X) = \mathbb{E}(\mathrm{co}(X))$, and $h(u, \mathbb{E}(X)) = \mathbb{E}(h(u,X))$ for all $u\in E$, where $h$ is the support function.  This support function characterization is powerful, and allows us to prove several properties about the selection expectation.  More importantly, the following results allow us to operationalize the selection expectation, which is useful from a practical standpoint for performing statistical analysis.

\begin{proposition}
Suppose $C,D$ are bounded and closed integrable random sets, and let $\Psi$ be a random matrix or a random scalar.  If the probability space is nonatomic, then
\begin{enumerate}[$\mathrm{(\mathtt{\alph*})}$, leftmargin=2.5em]
\item $\mathbb{E}(C) = \mathrm{co}(C)$, when $C$ is deterministic
\item $\mathbb{E}(C \oplus D) = \mathbb{E}(C)\oplus\mathbb{E}(D)$
\item $\mathbb{E}(\Psi C) = \mathbb{E}(\Psi)\cdot\mathbb{E}(C)$, when $\Psi$ is independent of $C$
\item $\mathbb{E}(C) \subseteq\mathbb{E}(D)$, when $C \subseteq D$ a.s.
\item $\mathbb{E}(C)\cup\mathbb{E}(D)\subseteq\mathbb{E}(C\cup D)$
\item $\mathbb{E}(C\cap D)\subseteq\mathbb{E}(C)\cap\mathbb{E}(D)$
\item $\mathbb{E}(C\ominus D) \subseteq \mathbb{E}(C)\ominus\mathbb{E}(D)$, when $C\ominus D$ is a.s. non-empty.
\end{enumerate}
\end{proposition}

\begin{proof}
The first result holds since $\mathbb{E}(X) = \mathbb{E}(\mathrm{co}(X))$ and $h(u, \mathbb{E}(C)) = \mathbb{E}(h(u,C)) = h(u,C)$.  The next result follows from $h(u, C\oplus D) = h(u,C) + h(u,D)$ \cite{schneider1993}, since $h(u, \mathbb{E}(C\oplus D)) = \mathbb{E}(h(u,C\oplus D)) = \mathbb{E}(h(u,C) + h(u,D)) = \mathbb{E}(h(u,C)) + \mathbb{E}(h(u,D)) = h(u,\mathbb{E}(C)) + h(u,\mathbb{E}(D)) = h(u,\mathbb{E}(C)\oplus\mathbb{E}(D))$.  The fourth result holds since $h(u, C) \leq h(u,D)$ when $C \subseteq D$ \cite{schneider1993}, which implies $h(u, \mathbb{E}(C)) = \mathbb{E}(h(u,C)) \leq \mathbb{E}(h(u,D)) = h(u, \mathbb{E}(D))$.  For the fifth result, note $C \subseteq C\cup D$ and $D\subseteq C\cup D$.  The fourth result gives $\mathbb{E}(C)\subseteq\mathbb{E}(C\cup D)$ and $\mathbb{E}(D)\subseteq\mathbb{E}(C\cup D)$, which implies $\mathbb{E}(C)\cup\mathbb{E}(D)\subseteq\mathbb{E}(C\cup D)$.  The sixth result follows since combining $C\cap D \subseteq C$, $C\cap D \subseteq D$, and the fourth result gives: $\mathbb{E}(C\cap D)\subseteq\mathbb{E}(C)$ and $\mathbb{E}(C\cap D)\subseteq\mathbb{E}(D)$, which implies $\mathbb{E}(C\cap D)\subseteq\mathbb{E}(C)\cap\mathbb{E}(D)$.  To prove the seventh result, note $(C\ominus D)\oplus D\subseteq C$ \cite{schneider1993}.  Applying the second and fourth results yields $\mathbb{E}(C\ominus D)\oplus\mathbb{E}(D) \subseteq\mathbb{E}(C)$, and so $\mathbb{E}(C\ominus D) \subseteq \mathbb{E}(C)\ominus\mathbb{E}(D)$ \cite{schneider1993}.



The third result cannot be proved using support functions since $h(x,\Psi C)$ cannot be written in terms of $h(x,C)$.  (If $\Psi = -1$, then $h(x,\Psi C) = \inf_{y\in C} x^\textsf{T}y$ while $h(x,C) = \sup_{y\in C} x^\textsf{T}y$.)  Our approach is to show $\mathcal{S}^1(\Psi C) = \Psi\mathcal{S}^1(C)$, since this implies $\mathbb{E}(\Psi C) = \{\mathbb{E}(\Psi\xi) : \xi\in\mathcal{S}^1(C)\} = \{\mathbb{E}(\Psi)\cdot\mathbb{E}(\xi) : \xi\in\mathcal{S}^1(C)\} = \mathbb{E}(\Psi)\cdot\{\mathbb{E}(\xi) : \xi\in\mathcal{S}^1(C)\} = \mathbb{E}(\Psi)\cdot\mathbb{E}(C)$. The inclusion $\mathcal{S}^1(\Psi C) \supseteq \Psi\mathcal{S}^1(C)$ is obvious by definition.  To prove the reverse inclusion, let $\{\xi_n, n\geq 1\}$ with $\xi_n\in\mathcal{S}^1(C)$ be the Castaing representation \cite{castaing1967multi,molchanov2006,rockafellar2009} of $C$.  Then $\{\Psi\xi_n, n\geq 1\}$ is the Castaing representation of $\Psi C$.  But by Lemma 1.3 of \cite{molchanov2006}, each selection in $\mathcal{S}^1(\Psi C)$ can be approximated arbitrarily well by step functions with arguments from $\{\Psi\xi_n, n\geq 1\}$.  Thus $\mathcal{S}^1(\Psi C)\subseteq \Psi\mathcal{S}^1(C)$, and so $\mathcal{S}^1(\Psi C) = \Psi\mathcal{S}^1(C)$ since both inclusions were shown.
\qed
\end{proof}

\begin{remark}
Note the assumptions for part (c) include the cases where: $\Psi$ is deterministic, $C$ is deterministic, or $\Psi$ has positive or negative entries.
\end{remark}

Another result used in statistics is Jensen's inequality \cite{bickel2006}, which bounds changing the order of applying an expectation and a convex function to a random variable.  Our next result shows we can generalize Jensen's inequality to set-valued functions.

\begin{proposition}[Jensen's Inequality]
Let $S(u)$ be a graph-convex set-valued function (i.e., $S((1-\lambda)u_0 + \lambda u_1) \supseteq (1-\lambda)\cdot S(u_0) + \lambda\cdot S(u_1) \text{ for }\lambda\in(0,1))$, and let $X$ be bounded and closed integrable random set.  If $S(\cdot)$ is locally bounded (i.e., $S(B)$ is bounded for every bounded set $B$) and continuous, then we have $S(\mathbb{E}(X)) \supseteq\mathbb{E}(S(X))$.\end{proposition}

\begin{proof}
The selection expectation equals the Debreu expectation under our assumptions \cite{molchanov2006}.  This means there exists a sequence of random sets $X_n$ with the distribution
\begin{equation}
X_n = F_{in} \text{ with probability } p_{in},\  \text{ for } i = 1,\ldots,n, \ \text{ with } \textstyle\sum_{i=1}^np_{in}=1
\end{equation}
such that $\aslim X_n = X$, $\mathbb{E}(X) = \lim_n \mathbb{E}(X_n)$, and $\mathbb{E}(X_n) = \bigoplus_{i=1}^n p_{in}\cdot F_{in}$.  Using the semicontinuous mapping theorem implies $\aslim_n S(X_n) = S(X)$, and so we have equality of the selection expectation and Debreu expectation \cite{molchanov2006}. This means that $\mathbb{E}(S(X)) = \lim_n \mathbb{E}(S(X_n))$ and $\mathbb{E}(S(X_n)) = \bigoplus_{i=1}^n p_{in}\cdot S(F_{in})$.  Next note $S(\bigoplus_{i=1}^n p_{in}\cdot F_{in}) \supseteq \bigoplus_{i=1}^n p_{in}\cdot S(F_{in})$ by the graph-convexity of $S(\cdot)$.  Taking the limit of this set relationship gives $S(\mathbb{E}(X)) = \lim S(\bigoplus_{i=1}^n p_{in}\cdot F_{in}) \supseteq \lim_n \bigoplus_{i=1}^n p_{in}\cdot S(F_{in}) = \mathbb{E}(S(X))$, where we have used the fact that $\lim_n S(\bigoplus_{i=1}^n p_{in}\cdot F_{in})= S(\mathbb{E}(X))$ by definition of the continuity of the set-valued function $S(\cdot)$.
\qed
\end{proof}

\begin{remark}
Jensen's inequality is sometimes stated for concave functions, and such a generalization exists for set-valued mappings.  If $S(u)$ is a graph-concave set-valued function (i.e., $S((1-\lambda)u_0 + \lambda u_1) \subseteq (1-\lambda)\cdot S(u_0) + \lambda\cdot S(u_1) \text{ for }\lambda\in(0,1))$) and the other assumptions of the above theorem hold, then we have $S(\mathbb{E}(X)) \subseteq\mathbb{E}(S(X))$.
\end{remark}


Lastly, we present a strong law of large numbers (\textsf{SLLN}) for the selection expectation.  The key idea is the Minkowski sum takes the role of averaging.

\begin{theorem}[Artstein and Vitale, 1975 \cite{artstein1975}]
\label{thm:slln}
Suppose the probability space is non-atomic.  If $X, X_i$, $i \geq 1$, are i.i.d. bounded and closed integrable random sets, then we have that: $\aslim_n \frac{1}{n}\bigoplus_{i=1}^n X_i = \mathbb{E}(X)$.
\end{theorem}

This particular strong law of large numbers can be generalized in a number of ways, and a survey of the different generalizations possible  can be found in \cite{molchanov2006}.

\subsection{Central Limit Theorems}

Unlike laws of large numbers that relate convergence of Minkowski sums of i.i.d. random sets $\frac{1}{n}\bigoplus_{i=1}^n X_i$ to their selection expectation $\mathbb{E}(X)$, analogs of the central limit theorem (\textsf{CLT}) relating Minkowski sums and selection expectations are less well-understood.  One major impediment is that the $\ominus$ operator does not generally invert the $\oplus$ operator, which means it is generally not possible to normalize (in the sense of having a zero mean) the Minkowski sum $\frac{1}{n}\bigoplus_{i=1}^n X_i$.  As a result, the standard approach to generalizing the central limit theorem is to normalize by instead considering the Hausdorff distance between Minkowski sum and the selection expectation.

\begin{theorem}[Weil, 1982 \cite{weil1982}]
Suppose the probability space is nonatomic.  If $X, X_i$, $i \geq 1$, are i.i.d. bounded and closed integrable random sets, then we have that: $\sqrt{n}\cdot\bbd_\infty(\frac{1}{n}\bigoplus_{i=1}^n X_i,\mathbb{E}(X)) \rightarrow \sup_u\{\|\zeta(u)\|\ |\ \|u\|\leq 1\}$ in distribution, where $\zeta(u)$ for $\|u\|\leq 1$ is a centered Gaussian random field with covariance given by: $\mathbb{E}(\zeta(u)\cdot \zeta(v)) = \mathbb{E}(h(u,X)\cdot h(v,X)) - \mathbb{E}(h(u,X))\cdot \mathbb{E}(h(v,X))$.
\end{theorem}

The difficulty with this central limit theorem is that it lacks a clear geometrical interpretation (in contrast to the classical central limit theorem for random variables) for the limiting distribution, and the question of whether such a geometrical interpretation exists remains open \cite{molchanov2006}.  However, one advantage of this formulation is that it lends itself to a generalization of the Delta method \cite{bickel2006} from statistics.


\begin{proposition}[Approximate Delta Method]
Suppose $r_n\bbd_\infty(C_n,C) \rightarrow w$ in distribution, where $r_n$ is a strictly increasing sequence, $C_n\in\mathcal{K}$ is a sequence of random sets, $C\in\mathcal{K}$ is a deterministic set, and $w$ is a random variable.  If $S$ is a Lipschitz continuous set-valued set function, then
\begin{equation}
\textstyle\limsup_n \mathbb{P}(r_n\bbd_\infty(S(C_n),S(C)) \geq u) \leq \mathbb{P}(\kappa \cdot w\geq u),
\end{equation}
where $\kappa\in\mathbb{R}_+$ is the Lipschitz constant of $S$.
\end{proposition}

\begin{proof}
Lipschitz continuity of $S$ gives $r_n\bbd_\infty(S(C_n),S(C)) \leq \kappa\cdot r_n\bbd_\infty(C_n,C)$. Thus $\mathbb{P}(r_n\bbd_\infty(S(C_n),S(C)) \geq u) \leq \mathbb{P}(\kappa\cdot r_n\bbd_\infty(C_n,C) \geq u)$.  The limit superior of both sides gives the result since $r_n\bbd_\infty(C_n,C) \rightarrow w$ in distribution. \qed
\end{proof}

\begin{remark}
The Delta method relates asymptotic distributions of random variables under differentiable functions \cite{bickel2006}, and the intuition is the derivative is used as a local approximation of the function.  The above result demonstrates one instance where Lipschitz continuity can be used as a local approximation for set-valued mappings.
\end{remark}

Though $\ominus$ does not generally invert $\oplus$, there is one special case when inversion is possible. If $C, D$ are compact convex sets, then $(C\oplus D)\ominus D = C$ \cite{schneider1993}.  Using this property, we describe a new central limit theorem for random sets with a particular structure that is useful for statistical applications.  Specifically, this result applies to randomly translated sets (\textsf{RaTS}), which are random sets of the form $C = K \oplus \xi$, where $K$ is a deterministic compact convex set, and $\xi$ is a (vector-valued) random variable.


\begin{theorem}[Central Limit Theorem for \textsf{RaTS}]
Suppose the probability space is nonatomic, and that $X, X_i$, $i \geq 1$, are i.i.d. random sets with $X_i = K \oplus \xi_i$, where $K$ is a deterministic compact convex set and $\xi,\xi_i$, $i\geq 1$, are i.i.d. (vector-valued) random variables with zero mean and finite variance. Then
\begin{equation}
\textstyle\sqrt{n}\cdot((\frac{1}{n}\bigoplus_{i=1}^n X_i)\ominus\mathbb{E}(X))\rightarrow\mathcal{N}(0,\mathbb{E}(\xi\xi^\mathsf{T} ))
\end{equation}
in distribution, where $\mathcal{N}(0,\mathbb{E}(\xi\xi^\mathsf{T} ))$ is a jointly Gaussian random variable with zero mean and covariance matrix given by $\mathbb{E}(\xi\xi^\mathsf{T} )$.
\end{theorem}

\begin{proof}
Since $(\frac{1}{n}\bigoplus_{i=1}^n X_i)\ominus\mathbb{E}(X) = ((\frac{1}{n}\sum_{i=1}^n\xi_i) \oplus K) \ominus K = \frac{1}{n}\sum_{i=1}^n\xi_i$, the result follows by the classical central limit theorem \cite{bickel2006}.\qed
\end{proof}

The benefit of this new formulation of the central limit theorem is that it has a clear geometrical interpretation like the classical central limit theorem for random variables, but unfortunately this result only applies to the specific class of \textsf{RaTS}.  

%
%
%

\section{Kernel Regression}

We will construct a nonparametric estimator for set-valued functions using an approach that can be viewed as a natural generalization of kernel regression methods for functions \cite{aswani2011,aswani2013,bickel1982,noda1976,wand1994}.  These techniques are considered nonparametric because, in contrast to parametric models with a finite number of parameters, the number of parameters in nonparametric models increases as the amount of data increases.

\subsection{Problem Setup}
\label{sec:krps}

Consider a Lipschitz continuous set-valued function $S(u) : U \rightrightarrows \mathbb{R}^q$ with random samples $(X_i, S_i)\in U\times\mathbb{R}^q$ for $i=1,\ldots,n$, where: $U\subseteq\mathbb{R}^d$ is a convex compact set; $S(u)$ is a convex compact set for each $u\in U$; $X_i$ are i.i.d. (vector-valued) random variables with a Lipschitz continuous density function $f_X$ that has the property $f_X(u)>0$ for $u\in U$; and $S_i = S(X_i) \oplus W_i$ with $W_i$ i.i.d. (vector-valued) random variables that have zero mean $\mathbb{E}(W)=0$ and finite variance $\|\mathbb{E}(WW^\mathsf{T} )\| < +\infty$.  The problem is to estimate $S(u)$ at any $u\in U$ using the above described samples, and we need convexity of $U$ to ensure its tangent cone is derivable at $\partial U$ \cite{rockafellar2009}; however, our results will hold for all $u\in\mathrm{int}(U)$ unconditional of any such regularity assumptions.

\subsection{Kernel Functions}

Kernel regression is so named because these approaches use kernel functions $\varphi : \mathbb{R}\rightarrow\mathbb{R}$, which are functions that are non-negative, bounded, even (i.e., $\varphi(-u) = \varphi(u)$), and have finite support (i.e., there is a constant $\eta \in (0,1)$ such that $\varphi(u) > 0$ when $|u| \leq \eta$, and $\varphi(u) = 0$ for $|u| \geq 1$).  One example of a kernel function is the indicator function $\varphi(u) = (1/2)\cdot\mathbf{1}(|u|< 1)$, and another example is the Epanechnikov kernel $\varphi(u) = (3/4)\cdot(1-u^2)\cdot\mathbf{1}(u \leq 1)$.  Notationally, it is useful to define the family of kernel functions $\varphi_h(u) = h^{-d}\varphi(\|u\|/h)$ and the function $\gamma(u) = \int_{z\in T_U(u)} \varphi_1(z)dz$, where $T_U(u)$ is the tangent cone of $U$ at the point $u$.  (Note $\gamma(u)$ is strictly greater than zero and finite because of the assumptions.) We first prove a lemma about $\varphi_h(u)$:

\begin{lemma}
\label{lem:kernel}
If $h = n^{-1/(d+4)}$, then for $u\in U$ we have
\begin{enumerate}[$\mathrm{(\mathtt{\alph*})}$, leftmargin=2.5em]
\item $\aslim_n\frac{1}{n}\sum_{i=1}^n \varphi_h(X_i-u) = \gamma(u)\cdot f_X(u)$ \label{lem:kernel_den}
\item $\vphantom{\displaystyle\sum}$$\aslim_n \frac{1}{n}\sum_{i=1}^n \varphi_h(X_i-u)\cdot w_i = 0$ \label{lem:kernel_noise}
\item $\aslim_n \frac{1}{n}\sum_{i=1}^n \varphi_h(X_i-u)\cdot\|X_i-u\| = 0$ \label{lem:kernel_ball}
\end{enumerate}
\end{lemma}

\begin{proof}
We prove these three results by verifying the hypothesis of Kolmogorov's strong law of large numbers holds in each case, then applying this law of large numbers, and finally computing the expectation of the corresponding quantity in each case. To prove the first result, observe that
\begin{equation}
\textstyle\lim_n\sum_{i=1}^n n^{-2}\cdot\mathrm{var}\big(\varphi_h(X_i-u)\big) \leq \lim_n c/\big(nh^d\big) < \infty,
\end{equation}
where the first inequality holds for some constant $c\in\mathbb{R}^+$ because $\varphi_h(X_i-u)$ is bounded and nonzero with probability at most $s\cdot h^d$ for some constant $s\in\mathbb{R}^+$; and the second inequality holds because $nh^d = n^{4/(d+4)}$.  The finiteness of the above summation means we can apply Kolmogorov's strong law of large numbers, which gives $\aslim_n \frac{1}{n}\sum_{i=1}^n \varphi_h(X_i-u) - \mathbb{E}(\frac{1}{n}\sum_{i=1}^n \varphi_h(X_i-u)) = 0$.  Our next step is to compute this expectation.  Note that
\begin{equation}
\begin{aligned}
\textstyle\mathbb{E}(\frac{1}{n}\sum_{i=1}^n \varphi_h(X_i-u)) & \textstyle= \int_{x\in\mathbb{R}^d} \varphi_h(x-u)\cdot f_X(x)dx \\
&\textstyle= \int_{z\in\mathbb{B}\cap(U\oplus\{-u\})/h} \varphi_1(z)\cdot f_X(u + hz)dz
\end{aligned}
\end{equation}
where in the last line we made the change of variables $z = (x-u)/h$.  Let $R(h) = (\mathbb{B}\cap(U\oplus\{-u\})/h)\setminus T_U(u)$ and $S(h) = (\mathbb{B}\cap T_U(u))\setminus(U\oplus\{-u\})/h$. So we have
\begin{equation}
\label{eqn:limres1}
\begin{aligned}
&\textstyle\big|\mathbb{E}(\frac{1}{n}\sum_{i=1}^n \varphi_h(X_i-u)) - f_X(u)\cdot\int_{z\in \mathbb{B}\cap T_U(u)} \varphi_1(z)dz\big| \\
&\qquad \leq\textstyle \int_{z\in\mathbb{B}} \varphi_1(z)\cdot \big|f_X(u + hz)-f_X(u)\big|dz + \int_{z\in R(h)\cup S(h)}|\varphi_1(z)f_X(u+hz)|dz\\
&\qquad \leq\textstyle \int_{z\in\mathbb{B}} \varphi_1(z)\cdot \kappa h\|z\|dz + s\int_{z\in R(h)}dz + s\int_{z\in S(h)}dz\\
&\qquad \leq\textstyle h\cdot\kappa\int_{z\in\mathbb{B}} \varphi_1(z) dz + s\int_{z\in R(h)}dz + s\int_{z\in S(h)}dz\\
\end{aligned}
\end{equation}
where $\kappa\in\mathbb{R}_+$ is the Lipschitz constant of the density $f_X(u)$, and $s\in\mathbb{R}_+$ is a constant that exists by continuity of $f_X(u)$.  Next note $s\int_{z\in R(h)}dz + s\int_{z\in S(h)}dz \rightarrow 0$ as $h\rightarrow 0$ by Proposition 6.2 and Theorem 4.10 of \cite{rockafellar2009}.  Thus taking the limit of (\ref{eqn:limres1}) gives $\lim_n \mathbb{E}(\frac{1}{n}\sum_{i=1}^n \varphi_h(X_i-u)) = f_X(u)\cdot\int_{z\in\mathbb{R}^d} \varphi_1(z)dz$.  This proves the first result when combined with the implication of Kolmogorov's strong law of large numbers in our setting, and after noting $\gamma(u) = \int_{z\in \mathbb{B}\cap T_U(u)} \varphi_1(z)dz$ since $\varphi_1(u) = 0$ for $\|u\| > 1$.

For the proof of the second result, let $\langle w\rangle_j$ denote the $j$-th component of the vector $w$.  Next observe that
\begin{equation}
\textstyle\lim_n\sum_{i=1}^n n^{-2}\cdot\mathrm{var}\big(\varphi_h(X_i-u)\cdot\langle w_i\rangle_j\big) \leq \lim_nc\cdot\mathrm{var}\big(\langle w\rangle_j\big)/\big(nh^d\big) < \infty,
\end{equation}
where the first inequality holds for some constant $c\in\mathbb{R}^+$ because the $w_i$ have zero mean and because $\varphi_h(X_i-u)$ is bounded and nonzero with probability at most $s\cdot h^d$ for some constant $s\in\mathbb{R}^+$; and the second inequality holds because $nh^d = n^{4/(d+4)}$.  The finiteness of the above summation means Kolmogorov's strong law of large numbers gives $\aslim_n \frac{1}{n}\sum_{i=1}^n \varphi_h(X_i-u)\cdot\langle w_i\rangle_j - \mathbb{E}(\frac{1}{n}\sum_{i=1}^n \varphi_h(X_i-u)\cdot\langle w_i\rangle_j) = 0$.  But the $w_i$ are zero mean, and so we have that $\mathbb{E}(\frac{1}{n}\sum_{i=1}^n \varphi_h(X_i-u)\cdot\langle w_i\rangle_j) = 0$.

To prove the third result, observe that
\begin{equation}
\textstyle\lim_n\sum_{i=1}^n n^{-2}\cdot\mathrm{var}\big(\varphi_h(X_i-u)\cdot\|X_i-u\|\big) \leq \lim_n c/\big(nh^d\big) < \infty,
\end{equation}
where the first inequality holds for some constant $c\in\mathbb{R}^+$ because $U$ is a compact set and because $\varphi_h(X_i-u)$ is bounded and nonzero with probability at most $s\cdot h^d$ for some constant $s\in\mathbb{R}^+$; and the second inequality holds because $nh^d = n^{4/(d+4)}$.  The finiteness of the above summation means we can apply Kolmogorov's strong law of large numbers, which gives $\aslim_n \frac{1}{n}\sum_{i=1}^n \varphi_h(X_i-u)\cdot\|X_i-u\| - \mathbb{E}(\frac{1}{n}\sum_{i=1}^n \varphi_h(X_i-u)\cdot\|X_i-u\|) = 0$.  Our next step is to compute this expectation.  Note that
\begin{equation}
\begin{aligned}
\textstyle\mathbb{E}(\frac{1}{n}\sum_{i=1}^n \varphi_h(X_i-u)\cdot\|X_i-u\|) & \textstyle\leq \int_{x\in\mathbb{R}^d} \varphi_h(x-u)\cdot\|x-u\|\cdot f_X(x)dx \\
&\textstyle\leq h\int_{z\in\mathbb{R}^d} \varphi_1(z)\cdot z\cdot f_X(u + hz)dz\\
&\textstyle\leq c\cdot h = c\cdot n^{-1/(d+4)}
\end{aligned}
\end{equation}
where the second line makes the change of variables $z = (x-u)/h$, and the third line holds for some constant $c\in\mathbb{R}^+$ because the kernel has finite support and the density is continuous.  The above expectation is non-negative, and so $\lim_n \mathbb{E}(\frac{1}{n}\sum_{i=1}^n \varphi_h(X_i-u)\cdot\|X_i-u\|) = 0$.  This proves the third result when combined with the outcome of Kolmogorov's strong law of large numbers.
\qed
\end{proof}


\subsection{Kernel Regression Estimator}

We define a kernel regression estimate of $S$ at the point $u$ to be
\begin{equation}
\label{eqn:kre}
\widehat{S}(u) = \textstyle\Big[\frac{1}{n}\bigoplus_{i=1}^n \varphi_h(X_i-u)\cdot S_i\Big]\cdot\Big[\frac{1}{n}\sum_{i=1}^n \varphi_h(X_i-u)\Big]^{-1}
\end{equation}
The following theorem proves the strong pointwise consistency of this estimator.

\begin{theorem}
If $h = n^{-1/(d+4)}$, then $\aslim_n\widehat{S}(u) =S(u)$ for $u \in U$.
\end{theorem}

\begin{proof}
Let $\kappa \in \mathbb{R}_+$ be the Lipschitz constant of $S$, and note that by Lipschitz continuity we have
\begin{multline}
\label{eqn:lipkern}
\textstyle\frac{1}{n}\bigoplus_{i=1}^n \varphi_h(X_i-u)\cdot\big(S(u)\oplus w_i\big) \subseteq \\
\textstyle\frac{1}{n}\bigoplus_{i=1}^n \varphi_h(X_i-u)\cdot\big(S_i \oplus \kappa\|X_i-u\|\mathbb{B}\big)\subseteq\\
\textstyle\frac{1}{n}\bigoplus_{i=1}^n \varphi_h(X_i-u)\cdot\big(S(u) \oplus 2\kappa\|X_i-u\|\mathbb{B}\oplus w_i\big)
\end{multline}
Corollary \ref{cor:minklim}\ref{cor:minklim_prod} and Lemma \ref{lem:kernel}\ref{lem:kernel_den} give $\aslim_n \frac{1}{n}\bigoplus_{i=1}^n \varphi_h(X_i-u)\cdot S(u) = \gamma(u)\cdot f_X(u)\cdot S(u)$, and Corollary \ref{cor:minklim}\ref{cor:minklim_osum} and Lemma \ref{lem:kernel}\ref{lem:kernel_noise} yield $\aslim_n\frac{1}{n}\bigoplus_{i=1}^n \varphi_h(X_i-u)\cdot(S(u)\oplus w_i) = \gamma(u)\cdot f_X(u)\cdot S(u)$.  Corollary \ref{cor:minklim}\ref{cor:minklim_prod} and Lemma \ref{lem:kernel}\ref{lem:kernel_ball} give $\aslim_n \frac{1}{n}\bigoplus_{i=1}^n \varphi_h(X_i-u)\cdot\kappa\|X_i-u\|\mathbb{B} = 0$, and so Corollary \ref{cor:minklim}\ref{cor:minklim_osum} implies $\aslim_n\frac{1}{n}\bigoplus_{i=1}^n \varphi_h(X_i-u)\cdot\big(S(u) \oplus 2\kappa\|X_i-u\|\mathbb{B}\oplus w_i\big) = \gamma(u)\cdot f_X(u)\cdot S(u)$. So applying the sandwich lemma to (\ref{eqn:lipkern}) yields $\aslim_n \textstyle\frac{1}{n}\bigoplus_{i=1}^n \varphi_h(X_i-u)\cdot\big(S_i \oplus \kappa\|X_i-u\|\mathbb{B}\big) = \gamma(u)\cdot f_X(u)\cdot S(u)$.  Corollary \ref{cor:invoplus} gives $\aslim_n \frac{1}{n}\bigoplus_{i=1}^n \varphi_h(X_i-u)\cdot S_i = \gamma(u)\cdot f_X(u)\cdot S(u)$. Finally, using Corollary \ref{cor:minklim}\ref{cor:minklim_inv} and Lemma \ref{lem:kernel}\ref{lem:kernel_den} imply that $\aslim_n\widehat{S}(u) =S(u)$.
\qed
\end{proof}


\subsection{Algorithms to Compute Kernel Regression Estimator}

The statistical consistency of our kernel regression estimator is a theoretical result, and numerical computation of this estimator using the measured data $(u_i, s_i)$ for $i=1,\ldots,n$ needs some discussion.  The key point is that the corresponding algorithm used to compute the estimator depends on the representation of the sets $s_i$.  Since the random sets $S_i$ are \textsf{RaTS}, we only need to consider different representations of convex sets.  Moreover, we focus our discussion on polytope representations since any compact convex set can be approximated arbitrarily well by polytopes \cite{schneider1993}.  

If the sets $s_i$ are each represented by polynomial time membership oracles, then 
\begin{equation}
\textstyle\frac{1}{n}\bigoplus_{i=1}^n \varphi_h(x_i-u)\cdot s_i = \Big\{\frac{1}{n}\bigoplus_{i=1}^n \varphi_h(x_i-u)\cdot t_i : t_i \in s_i\text{ for } i =1,\ldots,n\Big\},
\end{equation}
and so membership in the Minkowski sum can be determined in polynomial time.  Polynomial time membership oracles exist for $s_i$ in a known compact set $G$, with a self-concordant barrier function for $G$ and the functions defining $s_i$ \cite{nesterov1994}: The measurement of $s_i$ would consist of the function parameters defining $s_i$, and set membership is determined by using interior point to solve a feasibility problem.  Examples include polytopes $s_i = \{t_i : a_i t_i \leq b_i\}$, with measured data $a_i,b_i$; second-order cone sets $s_i = \{t_i: \|a_{i,j} t_i + b_{i,j}\|_2 \leq c_{i,j}^\textsf{T}t_i + d_{i,j} \text{ for } j = 1,\ldots,k\}$, with measured data $a_{i,j},b_{i,j},c_{i,j},d_{i,j}$; and combinations thereof.  Other examples can be found in \cite{nesterov1994}.

Next suppose the sets $s_i$ are each represented by the zonotopes $s_i = \bigoplus_{k=1}^p w_{ik}\cdot z_{k}$, where $w_{ik}$ are weights and $z_{k}$ are vectors, which are polytopes defined as the Minkowski sum of vectors.  Restated, the observations are the $w_{ik}$ and $z_{k}$.  Then 
\begin{equation}
\textstyle\frac{1}{n}\bigoplus_{i=1}^n \varphi_h(x_i-u)\cdot s_i = \bigoplus_{k=1}^p \big[\textstyle\frac{1}{n}\sum_{i=1}^n \varphi_h(x_i-u)\cdot w_{ik}\big]\cdot z_k,
\end{equation}
and so the Minkowski sum is polynomial time computable for this representation.

Lastly, suppose the sets $s_i$ are represented by the convex hull of a finite set of $p_i$ vertices, meaning that $s_i = \mathrm{co}(\{v_{i1},\ldots,v_{ip_i}\})$.  In this setting the measurements are the vertices of each set $s_i$, and the Minkowski sum is given by
\begin{multline}
\label{eqn:vms}
\textstyle\frac{1}{n}\bigoplus_{i=1}^n \varphi_h(x_i-u)\cdot s_i = \mathrm{co}\big(\big\{\frac{1}{n}\sum_{i=1}^n \varphi_h(x_i-u)\cdot v_{ij_i} : \\\text{ for } j_i = 1,\ldots,p_i\text{ and }i=1,\ldots,n\big\}\big).
\end{multline}
This is a polynomial time computation since the number of vertices is finite.

\subsection{Numerical Example}

\begin{figure}
\centering
\includegraphics[trim=42pt 4pt 35pt 5pt,clip=true]{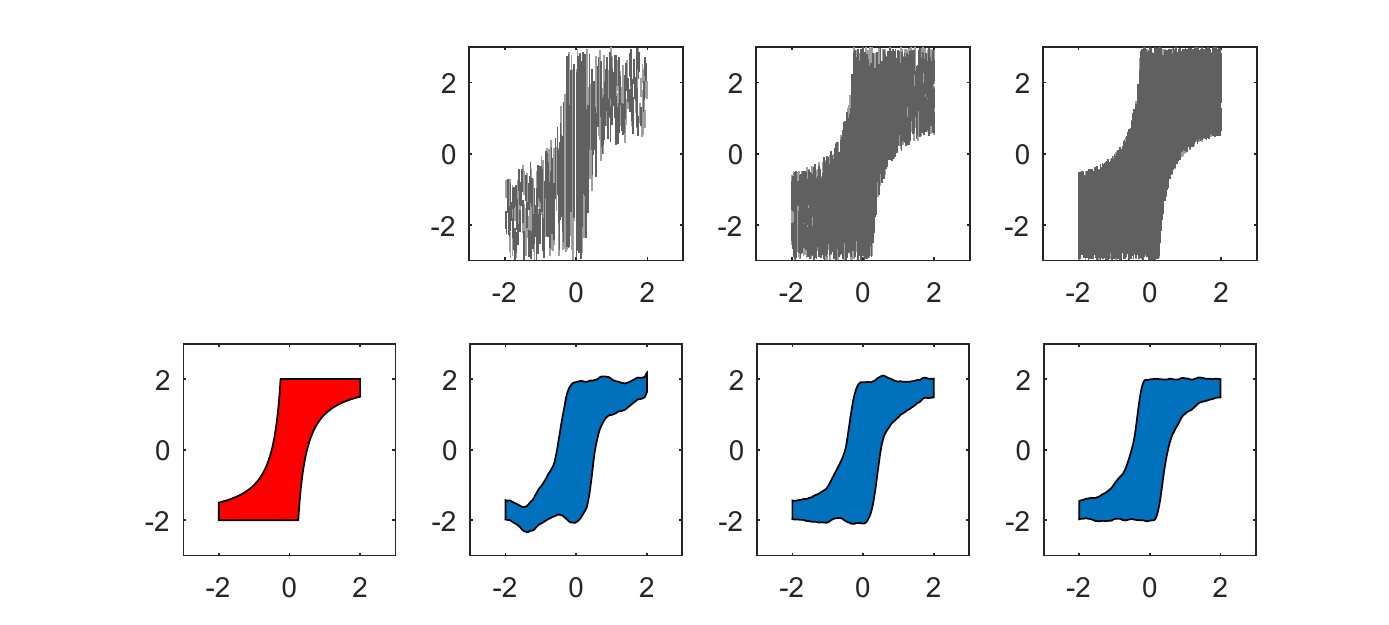}
\caption{The $x$-axes on the plots are $u$.  The top row from left to right shows the noisy measurements $(u_i,s_i)$ for $n = 10^2$, $n = 10^3$, and $n = 10^4$ data points, respectively, and the bottom row shows the set-valued function $S(u)$ being estimated and the corresponding estimates $\hat{S}(u)$ from our kernel regression estimator.\label{fig:svkr}}
\end{figure}

We conclude our discussion on kernel regression of set-valued functions with a numerical example to visually demonstrate the estimation problem being solved by our estimator.  Consider the set-valued function in the bottom-left of Fig. \ref{fig:svkr}, given by 
\begin{equation}
S(u) = \begin{cases}
\big[\hspace{0.67cm}-2, -\frac{2u+1}{u}+2\big], & \text{if } u \in \big[-2,-\frac{1}{4}\big]\\
\big[\hspace{0.67cm}-2,\hspace{1.32cm}2\big],&\text{if } u \in\big[-\frac{1}{4},\hspace{0.23cm}\frac{1}{4}\big]\\
\big[\frac{4u-1}{u}-2,\hspace{1.31cm}2\big],&\text{if }u\in\big[\hspace{0.36cm}\frac{1}{4},\hspace{0.27cm}2\big]
\end{cases}
\end{equation}
The $X_i$ variables have a $U(-2,2)$ distribution, and each measurement $s_i$ is in a vertex representation.  The noise $W_i$ has a $U(-1,1)$ distribution, meaning its variance is $1/6$.  The top row of Fig. \ref{fig:svkr} shows measurements for $n = 10^2$, $n = 10^3$, and $n = 10^4$ data points, respectively; and the bottom row shows estimates computed by (\ref{eqn:kre}) and (\ref{eqn:vms}) with an Epanechnikov kernel.\footnote{Our code \url{http://ieor.berkeley.edu/~aaswani/code/ssvf.zip} runs in a few seconds.}  This example shows that as the amount of data increases, the estimates $\hat{S}(u)$ converge pointwise to the actual set-valued function $S(u)$.

\section{Inverse Approximate Optimization}

Inverse optimization involves computing parameters that make measured solutions optimal \cite{ahuja2001,aswani2015,bertsimas2015,chan2014,esfahani2015,keshavarz2011}.  In contrast, the inverse approximate optimization problem makes noisy measurements of suboptimal solutions, and the goal is to estimate the amount of suboptimality and to estimate the parameters of optimization problem generating the data.  In principle, the \textsf{VIA} \cite{bertsimas2015} and \textsf{KKT} \cite{keshavarz2011} estimators can provide estimates of the desired quantities; but we show their estimates are statistically inconsistent.  As a result, we construct an estimator for inverse approximate optimization, prove its statistical consistency, and then discuss some possible generalizations.

\subsection{Problem Setup}

Consider a parametric convex optimization problem
\begin{equation}
\label{eqn:fop}
\textstyle V(u,\theta) = \min_x\big\{f(x,u,\theta)\ \big|\ g(x,u,\theta) \leq 0\big\}
\end{equation}
in which $f,g$ are continuous functions that are convex in $x$ for each fixed value of $u$ and $\theta$, and assume that for all $u,\theta$ the constraint qualification there exists $x$ such that $g(x,u,\theta) < 0$ holds.  (Note this constraint representation is fully general since we can write $g=\max g_i$.)  We use the definition that $\epsilon$-optimal solutions are those in the set
\begin{multline}
\label{eqn:def_eps}
\textstyle S(u,\epsilon,\theta) = \earg\min_x\big\{f(x,u,\theta)\ \big|\ g(x,u,\theta) \leq 0\big\} = \\
\{x : f(x,u,\theta) \leq V(u,\theta) + \epsilon, g(x,u,\theta) \leq 0\}.
\end{multline}
Our results also apply when $\epsilon$-optimal solutions are defined as in (\ref{eqn:def_eps}) but with $g(x,u,\theta) \leq \epsilon$.  The difference is (\ref{eqn:def_eps}) does not allow any constraint violation, while the alternative definition allows $\epsilon$ constraint violation.  Note there are other notions of $\epsilon$-optimal solutions like distance to the KKT graph, but we do not consider these.

Now suppose $\epsilon$-optimal solutions of (\ref{eqn:fop}) generate random samples $(U_i,Y_i)\in D\times\mathbb{R}^p$ for $i=1,\ldots,n$, where: $U_i$ are i.i.d. (vector-valued) random variables distributed on the set $D\subseteq\mathbb{R}^d$; $Y_i = X_i +  W_i$, where $X_i$ are i.i.d. (vector-valued) random variables distributed on $S(U_i,\epsilon_0,\theta_0)$ with constants $\epsilon_0\in\mathbb{R}_+$ and $\theta_0\in\mathbb{R}^p$; and $W_i$ are i.i.d. (vector-valued) random variables with zero mean $\mathbb{E}(W_i)=0$ and distributed on a known convex set $W$ with finite support (which implies finite variance).  We also assume the densities of $W_i,X_i$ are strictly positive on the interior of their supports (i.e., $f_W(u) > 0$ for $u\in\mathrm{int}(W)$ and $f_X(u|U_i) > 0$ for $u \in \mathrm{int}(S(U_i,\epsilon_0,\theta_0))$).

The inverse approximate optimization problem is to estimate $(\epsilon_0,\theta_0)$ using the $(U_i,Y_i)$ for $i=1,\ldots,n$.  Note that we assume the functional forms of $f,g$ are fixed.  Let $E \subseteq\mathbb{R}_+$ be a known closed set such that $\epsilon_0\in E$, and let $\Theta\subseteq\mathbb{R}^p$ be a known compact set such that $\theta_0\in\Theta$; the intuition is that these sets represent prior knowledge that constrain the parameters and amount of solution suboptimality.  The choice $E = \mathbb{R}_+$ corresponds to a situation with no such prior knowledge on $\epsilon_0$, and the compactness assumption on $\Theta$ is not restrictive in practice because this set can be made arbitrarily large.  (Unbounded $\Theta$ can also be used when a compactification with certain technical properties exists \cite{bahadur1971some}.) A so-called \emph{identifiability condition} \cite{bickel2006} is also needed.  We assume that if $(\epsilon_0,\theta_0) \in E\times\Theta$ then $\mathbb{E}(d^2(Y_i,S(U_i,\epsilon,\theta)\oplus W)) > 0$ for all $\epsilon\in [0,\epsilon_0)$ and $\theta\in\Theta\setminus\{\theta_0\}$.  An identifiability condition (such as the one we have assumed) intuitively says that different parameters of the model produce different outputs.

\subsection{Inconsistency of Existing Estimators}

The \textsf{VIA} \cite{bertsimas2015} (which minimizes the first order suboptimality of the data) and \textsf{KKT} \cite{keshavarz2011} (which minimizes the KKT suboptimality of the data) estimators are statistically inconsistent for $\epsilon_0 = 0$ \cite{aswani2015}, but since these approaches minimize the amount of suboptimality of the measured data it is initially unclear without further analysis whether these approaches are inconsistent for problem instances with $\epsilon_0 > 0$.  The following result provides qualitative insights into the behavior of these estimators.

\begin{proposition}
Let $r \in \mathbb{R}_+$ be a constant, and suppose $f = x^2$, $g = [x - 1; -x - 1]$, $\epsilon_0 = 1$, $W = \{w : \|w\| \leq r\}$, and $W_i,X_i$ are uniformly distributed.  Then estimates $\hat{\epsilon}$ generated by the \textsf{VIA} \cite{bertsimas2015} and \textsf{KKT} \cite{keshavarz2011} methods are such that $\asliminf_n \hat{\epsilon} > r/3$.
\end{proposition}

\begin{proof}
The \textsf{KKT} estimate is given by
\begin{multline}
\hat{\epsilon} = \textstyle\max\Big\{\frac{1}{n}\sum_{i=1}^n \langle g(Y_i)\rangle_1^+, \frac{1}{n}\sum_{i=1}^n \langle g(Y_i)\rangle_2^+, \frac{1}{n}\sum_{i=1}^n |2Y_i + \langle\Lambda_i\rangle_1 - \langle\Lambda_i\rangle_2|,\\
\textstyle \frac{1}{n}\sum_{i=1}^n |\langle\Lambda_i\rangle_1\cdot(Y_i-1)|, \frac{1}{n}\sum_{i=1}^n |\langle\Lambda_i\rangle_2\cdot(-Y_i-1)|\Big\}
\end{multline}
where these $\Lambda_i$ are the minimizers of the below optimization problem
\begin{multline}
\textstyle \min\big\{\frac{1}{n}\sum_{i=1}^n |2Y_i + \langle\lambda_i\rangle_1 - \langle\lambda_i\rangle_2| + \langle\lambda_i\rangle_1\cdot|Y_i-1| + \\
\textstyle\langle\lambda_i\rangle_2\cdot|-Y_i-1|\ \big|\ \lambda_i \geq 0, i = 1,\ldots,n\big\}.
\end{multline}
Since $\earg\min_x\{f(x,u,\theta)\ |\ g(x,u,\theta) \leq 0\} = [-1,1]$ under the hypothesis of this proposition, it holds the $Y_i$ are i.i.d. and have triangular distribution with lower limit $-r-1$, upper limit $r+1$, and mode $0$.  Hence the density of $C_i = \langle g(Y_i)\rangle_1^+$ is given by
\begin{equation}
f_C(u) = \textstyle\Big(\frac{1}{2} + \frac{1}{r+1} - \frac{1}{2(r+1)^2}\Big)\cdot\delta(u) + \Big(\frac{r}{(r+1)^2} - \frac{1}{(r+1)^2}\cdot u\Big)\cdot\mathbf{1}(u\in[0,r]),
\end{equation}
where $\delta(u)$ is the Dirac delta function.  So $\mathbb{E}(C_i) = \frac{r+1}{3}$, and the strong law of large numbers implies $\frac{r+1}{3} = \aslim_n \frac{1}{n}\sum_{i=1}^n C_i = \asliminf_n \frac{1}{n}\sum_{i=1}^n C_i \leq \asliminf_n \hat{\epsilon}$.

The \textsf{VIA} estimate is given by $\hat{\epsilon} = \frac{1}{n}\sum_{i=1}^n|\hat{\epsilon}_i|$ where $\epsilon_i$ are the minimizers to
\begin{equation}
\textstyle\min \Big\{\frac{1}{n}\sum_{i=1}^n|\hat{\epsilon}_i|\ \Big|\ 2Y_i(x_i-Y_i) \geq -\hat{\epsilon}_i \text{ for } x_i\in[-1,1], i=1,\ldots,n\Big\}
\end{equation}
However, observe that $2Y_i(x_i-Y_i) \geq -\hat{\epsilon}_i$ for $x_i\in[-1,1]$ simplifies to the constraint $-2(|Y_i|+Y_i^2) \geq -\hat{\epsilon}_i$.  Since the above optimization is minimizing each $|\hat{\epsilon}_i|$, this means the constraint will be $\hat{\epsilon}_i = 2(|Y_i|+Y_i^2)$ at optimality. Recall that as shown in the proof for \textsf{KKT}, the $Y_i$ have a triangular distribution with lower limit $-r-1$, upper limit $r+1$, and mode $0$.  This means $\mathbb{E}(|\hat{\epsilon}_i|) = \frac{2r+2}{3} + \frac{(r+1)^2}{9}$.  Applying the strong law of large numbers gives $\frac{2r+2}{3} + \frac{(r+1)^2}{9} = \aslim_n \frac{1}{n}\sum_{i=1}^n|\hat{\epsilon}_i| = \aslim_n \hat{\epsilon}$.
\qed
\end{proof}

This proposition shows that existing approaches cannot distinguish between noise in measurements versus suboptimality of the solutions.  The reason is that these approaches are minimizing an incorrect error metric: They minimize the amount of suboptimality of the measured data, and this is an incorrect error metric when the measured data is noisy because the noise increases the suboptimality of the measured data.  Moreover, this indistinguishability of existing approaches is unbounded in the sense that as the noise variance increases then their estimates of suboptimality increase in an unbounded way.  Such behavior is undesirable, and in fact the above result gives the following corollary on the statistical properties of \textsf{VIA} and \textsf{KKT}.

\begin{corollary}
The \textsf{VIA} \cite{bertsimas2015} and \textsf{KKT} \cite{keshavarz2011} estimators are statistically inconsistent.
\end{corollary}

\begin{proof}
By definition an estimator is consistent for a class of models if and only if it is consistent for each model in that class. Thus to show inconsistency of \textsf{VIA} and \textsf{KKT} it suffices to show inconsistency for a single model.  The above proposition establishes inconsistency of \textsf{VIA} and \textsf{KKT} for a particular model because $\epsilon_0 = 1$ while $\asliminf_n \hat{\epsilon} > r/3$, meaning these approaches are inconsistent when $r > 3$.
\qed
\end{proof}

\subsection{Approximate Bilevel Programming (\textsf{ABP}) Estimator}

To correct the indistinguishability (between suboptimality of solutions and noise in measurements) problem faced by existing approaches, we instead propose an estimator that explicitly models the measured data as consisting of a suboptimal solution added to noise.  More specifically, we propose the following statistical estimator
\begin{equation}
\label{eqn:abp}
\textstyle(\breve{\rule{0ex}{1.5ex}\epsilon},\breve{\theta}) \in\textstyle\arg\min \Big\{\frac{1}{n}\sum_{i=1}^nd^2(Y_i,S(U_i,\epsilon,\theta)\oplus W) + \lambda\cdot\epsilon\ \Big|\ \epsilon \in E, \theta\in\Theta\Big\}
\end{equation}
where $\lambda\in\mathbb{R}_+$ and $d^2$ is the squared distance function defined in the preliminaries.  It is also useful to consider estimators defined as approximate solutions to the above optimization problem.  Let $z\in\mathbb{R}_+$ be a nonnegative value, and define the estimates
\begin{multline}
\textstyle(\hat{\rule{0ex}{1.5ex}\epsilon},\hat{\theta}) \in \Big\{\epsilon \in E, \theta\in\Theta : \frac{1}{n}\sum_{i=1}^nd^2(Y_i,S(U_i,\epsilon,\theta)\oplus W) + \lambda\cdot\epsilon \leq \\\textstyle\frac{1}{n}\sum_{i=1}^nd^2(Y_i,S(U_i,\breve{\rule{0ex}{1.5ex}\epsilon},\breve{\theta})\oplus W) + \lambda\cdot\breve{\rule{0ex}{1.5ex}\epsilon} + z\Big\}.
\end{multline}
For notational convenience, we will call this estimator the \textsf{ABP} estimator.  Note these estimates are defined as being any $z\zarg\min$ of the optimization problem (\ref{eqn:abp}).

\begin{theorem}
\label{thm:abp_cons}
The \textsf{ABP} estimator is strongly statistically consistent, meaning we have $\aslim_n (\hat{\rule{0ex}{1.5ex}\epsilon},\hat{\theta}) = (\epsilon_0,\theta_0)$ whenever $\lambda = 1/n$ and $\lim_n (n\cdot z) = 0$.
\end{theorem}

\begin{proof}
Our first step is to show $d^2(y,S(u,\epsilon,\theta)\oplus W)$ satisfies certain continuity properties.  Note $\{x : g(x,u,\theta) \leq 0\}$ is continuous by Example 5.10 of \cite{rockafellar2009}, and so $V(u,\theta)$ is continuous by the Berge maximum theorem \cite{berge1963}.  Noting $d^2(y,S(u,\epsilon,\theta)\oplus W) = \min\{\|y-\hat{y}\|^2\ |\ \hat{y} = \hat{x} + \hat{\epsilon}, \hat{\epsilon} \in W, f(\hat{x},u,\theta) \leq V(u,\theta) + \epsilon, g(\hat{x},u,\theta) \leq 0\}$, we can apply the Berge maximum theorem \cite{berge1963} since this feasible set is osc by Example 5.8 of \cite{rockafellar2009}: This implies $d^2(y,S(u,\epsilon,\theta)\oplus W)$ is lower semicontinuous in $(\epsilon,\theta)$, and so $\mathbb{E}(d^2(Y_i,S(U_i,\epsilon,\theta)\oplus W))$ is lower semicontinuous in $(\epsilon,\theta)$ by Fatou's lemma.

Next note that the estimate $(\breve{\rule{0ex}{1.5ex}\epsilon},\breve{\theta})$ also minimizes the optimization problem
\begin{equation}
\label{eqn:abp_reform}
\textstyle\min \Big\{\sum_{i=1}^nd^2(Y_i,S(U_i,\epsilon,\theta)\oplus W) + n\lambda\cdot\epsilon\ \Big|\ \epsilon \in E, \theta\in\Theta\Big\}
\end{equation}
But $n\lambda = 1$ by assumption, and so the objective of (\ref{eqn:abp_reform}) is nondecreasing in $n$.  Hence $\mathbb{P}(\elim_n \sum_{i=1}^nd^2(Y_i,S(U_i,\epsilon,\theta)\oplus W) + \epsilon = \sup_n \sum_{i=1}^nd^2(Y_i,S(U_i,\epsilon,\theta)\oplus W) + \epsilon) = 1$ by Proposition 7.4 of \cite{rockafellar2009}, where $\elim$ is the epi-limit \cite{rockafellar2009}.  We next prove that
\begin{equation}
\label{eqn:suplimn}
\textstyle\mathbb{P}(\sup_n \sum_{i=1}^nd^2(Y_i,S(U_i,\epsilon,\theta) \oplus W) > 0 \text{ for } (\epsilon,\theta)\in([0,\epsilon_0]\times\Theta)\setminus\{(\epsilon_0,\theta_0)\})) = 1,
\end{equation}
and our approach is to use a well-known covering argument originally due to Wald \cite{wald1949note}.  Let $S_k$ be a decreasing sequence (i.e., $S_k\supseteq S_{k+1} \supseteq \cdots$) of open neighborhoods of $(\epsilon_0,\theta_0)$, with $\lim_k S_k = \{(\epsilon_0,\theta_0)\}$.  Since $\mathbb{E}(d^2(Y_i,S(U_i,\epsilon,\theta)\oplus W))$ is lower semicontinuous in $(\epsilon,\theta)$, this means $\min \{\mathbb{E}(d^2(Y_i,S(U_i,\epsilon,\theta)\oplus W))\ |\ (\epsilon,\theta) \in ([0,\epsilon_0]\times\Theta)\setminus S_k\} > 0$ by the identifiability condition.  Thus there exists $\nu_k > 0$ such that $\mathbb{E}(d^2(Y_i,S(U_i,\epsilon,\theta)\oplus W)) > 2\nu_k$ for $(\epsilon,\theta) \in ([0,\epsilon_0]\times\Theta)\setminus S_k$.  By lower semicontinuity of $d^2(y,S(u,\epsilon,\theta)\oplus W)$ and the monontone convergence theorem, there exists an open neighborhood $T_k(\epsilon,\theta)$ for each $(\epsilon,\theta)\in([0,\epsilon_0]\times\Theta)\setminus S_k$ so that we have $\mathbb{E}(\inf \{d^2(Y_i,S(U_i,\epsilon',\theta')\oplus W)\ |\ (\epsilon',\theta')\in T_k(\epsilon,\theta)) > \mathbb{E}(d^2(Y_i,S(U_i,\epsilon,\theta)\oplus W)) - \nu_k$.  Since $([0,\epsilon_0]\times\Theta)\setminus S_k$ is compact, there exists a finite set $F_k\in([0,\epsilon_0]\times\Theta)\setminus S_k$ such that $T_k(\epsilon,\theta)$ for $(\epsilon,\theta)\in F_k$ forms a finite subcover of $([0,\epsilon_0]\times\Theta)\setminus S_k$.  Combining the above with the Borel-Cantelli lemma implies $\mathbb{P}(\inf\{\sup_n \sum_{i=1}^nd^2(Y_i,S(U_i,\epsilon',\theta') \oplus W)\ |\ (\epsilon',\theta') \in T_k(\epsilon,\theta)\} > 0) = 1$ for each $(\epsilon,\theta)\in F_k$, which by the finiteness of $F_k$ implies that $\mathbb{P}(\sup_n \sum_{i=1}^nd^2(Y_i,S(U_i,\epsilon,\theta) \oplus W) > 0 \text{ for } (\epsilon,\theta)\in([0,\epsilon_0]\times\Theta)\setminus S_k)) = 1$.  The desired (\ref{eqn:suplimn}) follows since we choose the sequence $S_k$ such that $S_k\downarrow \{(\epsilon_0,\theta_0)\}$.



Next consider the optimization problem
\begin{equation}
\label{eqn:limopt}
\textstyle\min \Big\{\sup_n\sum_{i=1}^nd^2(Y_i,S(U_i,\epsilon,\theta)\oplus W) + \epsilon\ \Big|\ \epsilon \in E, \theta\in\Theta\Big\}
\end{equation}
Note $(\epsilon_0,\theta_0)$ is feasible for both (\ref{eqn:abp_reform}) and (\ref{eqn:limopt}), and so the minimums of (\ref{eqn:abp_reform}) and (\ref{eqn:limopt}) are both less than or equal to $\epsilon_0$.  This means $\epsilon > \epsilon_0$ cannot minimize (\ref{eqn:limopt}).  Furthermore, using (\ref{eqn:suplimn}) implies that almost surely the (unique) minimizer of (\ref{eqn:limopt}) is $(\epsilon_0,\theta_0)$, and almost surely the minimum value of (\ref{eqn:limopt}) is $\epsilon_0$.  But from the argument in the preceding paragraph, (\ref{eqn:abp_reform}) epi-converges almost surely to (\ref{eqn:limopt}) since $E,\Theta$ are fixed.  The result now follows from Theorem 7.33 of \cite{rockafellar2009}.
\qed
\end{proof}

The above result concerns almost sure convergence of the \textsf{ABP} estimates $(\hat{\rule{0ex}{1.5ex}\epsilon},\hat{\theta})$ to the actual parameters $(\epsilon_0,\theta_0)$, but a related question is whether the corresponding solution set estimates $S(u, \hat{\rule{0ex}{1.5ex}\epsilon},\hat{\theta})$ converge to the actual solution sets $S(u, \epsilon_0,\theta_0)$.  Our semicontinuous mapping theorem can be used to establish almost sure convergence of the solution set estimates, and this argument leads to the the following corollary.

\begin{corollary}
We have that $\aslimsup_n S(u,\hat{\rule{0ex}{1.5ex}\epsilon},\hat{\theta}) \subseteq S(u,\epsilon_0,\theta_0)$ for $u\in D$.  If $\epsilon_0 > 0$ or $f(\cdot,u,\theta)$ is strictly convex in $x$, then $\aslim_n S(u,\hat{\rule{0ex}{1.5ex}\epsilon},\hat{\theta}) = S(u,\epsilon_0,\theta_0)$ for $u\in D$.
\end{corollary}

\begin{proof}
The above proof established that $S(u,\epsilon,\theta)$ is osc in $\epsilon,\theta$.  And so the first part of the corollary follows by the semicontinuous mapping theorem.  If $\epsilon_0 > 0$ then $S(u,\epsilon_0,\theta_0)$ is continuous at $(\epsilon_0,\theta_0)$ by Example 5.10 of \cite{rockafellar2009}.  If $f(\cdot,u,\theta)$ is strictly convex in $x$, then $S(u,0,\theta_0)$ is single-valued \cite{rockafellar2009}.  Hence $S(u,0,\theta_0)$ is continuous because a single-valued, osc, and locally bounded function is continuous \cite{rockafellar2009}. Thus the second part of the corollary follows from the semicontinuous mapping theorem.\qed
\end{proof}

\subsection{Algorithms to Compute \textsf{ABP} Estimator}

We next discuss numerical computation of \textsf{ABP} using the data $(u_i, y_i)$ for $i = 1,\ldots,n$.  The \textsf{ABP} estimator is an approximate (i.e., the solution sets have $\epsilon$ possibly greater than zero) bilevel program, which are optimization problems where some decision variables are solutions to optimization problems that are called the lower level problem. One approach to solve bilevel programs replaces the lower level problem with its KKT conditions \cite{allende2013,dempe2012}, and this can sometimes be rewritten as mixed-integer programs that may be numerically solved quickly \cite{aswani2016_wl}.  Another approach upper bounds the objective function of the lower level problem by its value function \cite{outrata1990,ye1995}.

Here we describe how a third approach that upper bounds the objective function of the lower level problem by its dual function \cite{aswani2015,aswani2016} can be used to compute the \textsf{ABP} estimator.  If $h(u,\theta,\lambda)$ is the Lagrangian dual function corresponding to (\ref{eqn:fop}), then under mild conditions ensuring zero duality gap the \textsf{ABP} estimator is given by
\begin{equation}
\label{eqn:dual_reform}
\begin{aligned}
\textstyle(\breve{\rule{0ex}{1.5ex}\epsilon},\breve{\theta}) \in\textstyle\arg\min\ &\textstyle\frac{1}{n}\sum_{i=1}^n\|y_i - \hat{x}_i\|^2 + \lambda\cdot\epsilon\\
\text{s.t. } & f(\hat{x}_i, u_i, \theta) \leq h(u_i,\theta,\lambda_i) + \epsilon\\
&g(\hat{x}_i,u_i,\theta) \leq 0\\
&\lambda_i\geq0,\epsilon \in E, \theta\in\Theta
\end{aligned}
\end{equation}
This duality-based reformulation can be numerically solved by two different algorithms \cite{aswani2015,aswani2016}, which we briefly describe here.  More details can be found in the corresponding references, and both algorithms assume the sets $E,\Theta$ are compact.  

Since the reformulation (\ref{eqn:dual_reform}) is a convex optimization problem for fixed $(\epsilon,\theta)$, one algorithm \cite{aswani2015} for computing \textsf{ABP} is to: discretize the set $E\times\Theta$ into a finite set $\Delta = \{(\epsilon_1,\theta_1),\ldots,(\epsilon_k,\theta_k)\}$ such that it forms a set covering with balls of a prescribed radius, compute the minimum objective function value of (\ref{eqn:dual_reform}) for $(\epsilon,\theta)\in\Delta$ (which we call $Q(\epsilon,\theta)$), and then choose estimates $\textstyle(\hat{\rule{0ex}{1.5ex}\epsilon},\hat{\theta}) = \arg\min\{Q(\epsilon,\theta)\ |\ (\epsilon,\theta)\in\Delta\}$.  A result from \cite{aswani2015} implies that estimates chosen using this \emph{enumeration algorithm} satisfy the assumptions of Theorem \ref{thm:abp_cons}, which is sufficient for statistical consistency.

A second algorithm \cite{aswani2016} replaces the Lagrangian dual by a numerically computed dual.  Partial dualization is used to define a regularized dual function (\textsf{RDF})
\begin{equation}
h_\mu(u,\theta,\lambda) = \textstyle\min_x\big\{\mu\cdot\|x\|^2+f(x,u,\theta) + \lambda^\mathsf{T}g(x,u,\theta)\ |\ x \in X\}.
\end{equation}
Here, $X$ is any compact set defined such that $\{x : \exists (u,\theta)\in U\times\Theta \text{ s.t. } g(x,u,\theta) \leq 0\} \subseteq \mathrm{int}(X)$.  The intuition is that $X$ is a set that contains all the feasible sets of (\ref{eqn:fop}) within its interior.  When $g$ does not depend on $(u,\theta)$, we can choose $X = \{x : l_i-1 \leq x_i \leq u_i+1\}$ with $u_i = \max \{x_i\ |\ g(x) \leq 0\}$ and $l_i = \min\{x_i\ |\ g(x) \leq 0\}$ that are computed by solving convex optimization problems.  Many applications of inverse approximate optimization consist of such a setting where the feasible set is independent of the inputs $u$ or the parameters $\theta$.  The benefit of the \textsf{RDF} is it can be numerically computed because it is a convex optimization problem, and that its gradient
\begin{equation}
\begin{aligned}
\nabla_{\theta}h_\mu(u,\theta,\lambda) &= \nabla_{\theta}f(x,u,\theta) + \lambda^\mathsf{T}\cdot\nabla_{\theta}g(x,u,\theta)\\
\nabla_{\lambda}h_\mu(u,\theta,\lambda) &= g(x,u,\theta)
\end{aligned}
\end{equation}
always exists when $\mu > 0$.  In contrast, the Lagrangian dual is usually only directionally differentiable but not differentiable.  The algorithm proceeds by using a nonlinear numerical solver to solve a sequence of optimization problems in which $\mu$ goes to $0$.

A third possibility is a polynomial time approximation algorithm with the property that statistical consistency holds as the amount of samples $n$ increases to infinity.  Such an algorithm has been constructed, when $f$ is affine in $\theta$ and $g$ does not depend on $\theta$, for inverse optimization with noisy data \cite{aswani2015}; it uses kernel regression to pre-smooth the data and then solves a convex problem corresponding to inverse optimization assuming no noise in the pre-smoothed data.  Here we sketch a similar algorithm for inverse approximate optimization, and we leave its analysis for future work.  Define $\hat{S}(u) = \mathrm{co}(\{y_i : \|u_i - u\| \leq h\})\ominus W$ for $h\in\mathbb{R}_+$, and choose the data $\hat{x}_i$ by sampling from the uniform distribution on $\hat{S}(u_i)$.  The estimate $(\breve{\rule{0ex}{1.5ex}\epsilon},\breve{\theta})$ is computed by solving (\ref{eqn:dual_reform}) with the change that the $g(\hat{x}_i,u_i,\theta) \leq 0$ constraints are removed.

\subsection{Numerical Example}

\begin{figure}
\centering
\includegraphics[trim=42pt 27pt 35pt 21pt,clip=true]{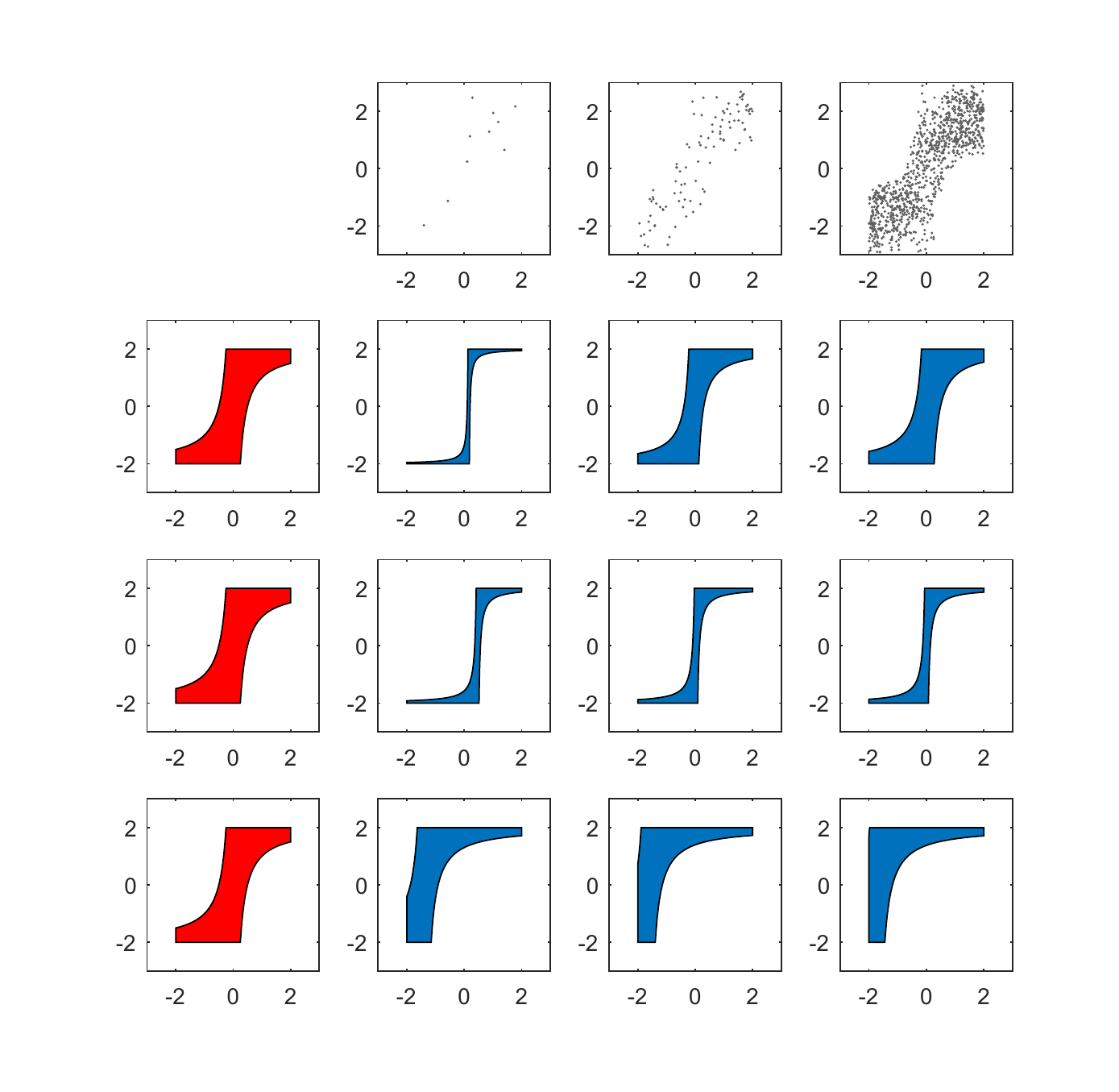}
\caption{The $x$-axes on the plots are $u$. For the solution set $S(u,\epsilon,\theta) = \earg\min_x\{-(\theta+u)\cdot x\ |\ -2\leq x\leq 2\}$, the left column shows the actual $S(u,\epsilon_0,\theta_0)$ when $\epsilon_0 = 1$ and $\theta_0 = 0$.  The top row from left to right shows the noisy measurements $(u_i,y_i)$ for $n = 10^1$, $n = 10^2$, and $n = 10^3$ data points, respectively, and the rows below show the estimated solution set as computed by \textsf{ABP}, \textsf{KKT}, and \textsf{VIA}, respectively.\label{fig:solset}}
\end{figure}

We next consider a numerical example to visually compare estimates of $S(u,\epsilon_0,\theta_0)$ produced by our \textsf{ABP} estimator and the \textsf{VIA} \cite{bertsimas2015} and \textsf{KKT} \cite{keshavarz2011} estimators.  Suppose $x\in\mathbb{R}$, $f = -(\theta+u)\cdot x$, $g = [x-2;-x-2]$, $\epsilon_0 = 1$, $\theta_0 = 0$, $W = \{w : \|w\|\leq 1\}$, $U_i$ has a uniform distribution $U(-2,2)$, $X_i$ is uniformly distributed on $S(U_i,\epsilon_0,\theta_0)$, $E = \{\epsilon : 0.1 \leq \epsilon \leq 10\}$, and $\Theta = \{\theta : -2 \leq \theta\leq 2\}$.  The solution set $S(u,\epsilon_0,\theta_0)$ in this setting is shown in the left column of Fig. \ref{fig:solset}. Each measurement $(u_i, y_i)$ for this example is a point, and the top row of Fig. \ref{fig:solset} shows the measurements for $n = 10^1$, $n = 10^2$, and $n = 10^3$ data points, respectively. The rows below show the estimated (using the measurements shown above) solution set as computed by \textsf{ABP}, \textsf{KKT}, and \textsf{VIA}, respectively\footnote{Our code \url{http://ieor.berkeley.edu/~aaswani/code/ssvf.zip} runs in about three hours.}.  This example shows that as the number of measurements increases, the solution set estimated by \textsf{ABP} (\textsf{KKT} and \textsf{VIA}) converges (does not converge) to the actual solution set.  This statistical behavior is expected given our theoretical results on the strong consistency of \textsf{ABP} and the statistical inconsistency of \textsf{KKT} and \textsf{VIA}.

\subsection{Related Inverse Optimization Problems}

In our problem setup, the measurement noise $W_i$ had a distribution with a finite support.  However, noise models commonly used in statistics include distributions with unbounded support but finite variance.  The canonical example is $W_i$ that are jointly Gaussian with zero mean and finite covariance.  A heuristic approach for distributions with unbounded support is to use our \textsf{ABP} estimator with the choices of $W = (2\log n)^{1/2}\cdot \Sigma$ for sub-Gaussian distributions (i.e., distributions bounded from above by a jointly Gaussian random variable) and $W = ((2\log n)^{1/2} + \log n)\cdot\Sigma$ for sub-exponential distributions (i.e., distributions with exponentially decaying tails), where $\Sigma = \mathbb{E}(W_i^{\vphantom{\mathsf{T}}}W_i^\mathsf{T})$ is the covariance matrix of $W_i$.  The reason for this suggested heuristic is these choices of $W$ are analogous to bounds on the maximum expected values of sub-Gaussian and sub-exponential random variables \cite{boucheron2013}.

Since the \textsf{ABP} estimator is a heuristic in this setting, an obvious topic is to design a statistically consistent estimator for inverse approximate optimization problems with unbounded noise.  Maximum likelihood estimation is arguably the most natural approach because otherwise it is difficult to distinguish between noise and suboptimality of solutions.  Specifically, consider the original problem setup but with the changes that the random sample is $(u_i,X_i)$, the $X_i$ are uniformly distributed within $S(u_i,\epsilon_0,\theta_0)$, and that $W_i$ is distributed according to some known density $f_W(u)$.  Then the maximum likelihood estimator (\textsf{MLE}) for this modified problem setup is given by 
\begin{multline}
\label{eqn:iopt_mle}
\textstyle(\epsilon_\mathrm{mle},\theta_\mathrm{mle}) \in \arg\min \Big\{-\frac{1}{n}\sum_{i=1}^n \log\int_{x \in S(u_i, \epsilon, \theta)}f_W(Y_i-x)dx + \\
\textstyle\frac{1}{n}\sum_{i=1}^n\log\int_{x \in S(u_i, \epsilon, \theta)}dx\ \Big|\ \epsilon \in E, \theta\in\Theta\Big\}.
\end{multline}
This optimization problem has a challenging structure in which the domains of integration depend upon the decision variables \cite{royset2017variational}, and presents an opportunity for the further study of designing numerical algorithms to solve such optimization problems.  We do note that for fixed $(\epsilon,\theta)$, the integrals in the objective can be numerically computed in polynomial time using hit-and-run techniques for sampling from convex sets \cite{lovasz2006,smith1984}.  And so the enumeration algorithm we described earlier for the \textsf{ABP} estimator could be easily modified to solve this \textsf{MLE} problem.

\begin{remark}
The \textsf{ABP} and \textsf{MLE} estimators are actually qualitatively the same.  The $\frac{1}{n}\sum_{i=1}^nd^2(Y_i,S(U_i,\epsilon,\theta)\oplus W)$ term in \textsf{ABP} and the $-\frac{1}{n}\sum_{i=1}^n \log\int_{x \in S(u_i, \epsilon, \theta)}f_W(Y_i-x)dx$ term in \textsf{MLE} both penalize estimates in which the solutions $Y_i$ are far from the solution sets $S(\cdot, \epsilon, \theta)$, and the $\frac{1}{n}\sum_{i=1}^n\log\int_{x \in S(u_i, \epsilon, \theta)}dx$ term in \textsf{MLE} and the $\lambda\cdot\epsilon$ term in \textsf{ABP} both penalize  estimates that generate large solution sets.
\end{remark}

In the two inverse approximate optimization problem setups considered above, we assumed the approximate solutions $X_i$ were drawn from the solution sets $S(U_i,\epsilon_0,\theta_0)$ according to some distribution.  However, another modified problem setup would be to assume the $X_i$ were chosen from the solution sets by solution of another optimization problem.  This kind of setup corresponds to a scenario in which the $X_i$ are solutions to an optimistic bilevel optimization problem with unique solutions:
\begin{equation}
x = \arg\min \big\{F(u,x,z)\ \big|\ x \in S(u,\epsilon_0,\theta_0), G(u,x,z) \leq 0\big\}.
\end{equation}
In this case, the estimation procedure can be posed as a least squares problem
\begin{equation}
\label{eqn:iopt_ble}
\begin{aligned}
(\epsilon_\mathrm{ble},\theta_\mathrm{ble}) \in \arg\min\ & \textstyle\frac{1}{n}\sum_{i=1}^n \|Y_i-x_i\|^2\\
\text{s.t. }& x_i = \arg\min \big\{F(U_i,x,z)\ \big|\ x \in S(U_i,\epsilon,\theta), G(U_i,x,z) \leq 0\big\}\\
&\epsilon\in E,\theta\in\Theta.
\end{aligned}
\end{equation}
This is a challenging multi-level optimization problem and presents an opportunity for the further study of designing numerical algorithms to solve such optimization problems.  We do note that for fixed $(\epsilon,\theta)$, this becomes a convex optimization problem.  And so the enumeration algorithm we described earlier for the \textsf{ABP} estimator could be easily modified to solve this least squares problem.

\section{Conclusion}

In this paper, we used variational analysis to develop tools for statistics with set-valued functions, and then applied these tools to two estimation problems.  We constructed and studied a kernel regression estimator for set-valued functions and an estimator for the inverse approximate optimization problem.  The area of statistics with set-valued functions remains largely unexplored with many remaining problems.  One question is the design of numerical representations of sets and set-valued functions.  Though constraint representations of sets are pervasive, numerical machinery like epi-splines \cite{royset2016} may offer greater representational flexibility.  Another question is the development of numerical algorithms to solve optimization problems that arise in statistical estimation for set-valued functions.  Related inverse optimization problems lead to formulations (\ref{eqn:iopt_mle}) and (\ref{eqn:iopt_ble}) with structures that are not well-studied from the perspective of numerical optimization.  Further study of statistics with set-valued functions will require developing new numerical methods and optimization theory.



\bibliographystyle{spmpsci}      
\bibliography{set_stat}   

\end{document}